\documentclass{amsart}
\usepackage{graphicx}
\newtheorem{thm}{Theorem}[section]
\newtheorem{cor}[thm]{Corollary}
\newtheorem{lem}[thm]{Lemma}
\newtheorem{prop}[thm]{Proposition}
\theoremstyle{definition}

\theoremstyle{remark}

\numberwithin{equation}{section}

\begin{document}

\title[Optimal control of a large dam]{Optimal control of a large dam,
taking into account the water costs}%
\author{Vyacheslav M. Abramov}%
\address{School of Mathematical Sciences, Monash University, Building 28M,
Wellington road, Clayton, VIC 3800, Australia}%
\email{vyacheslav.abramov@sci.monash.edu.au}%


\begin{abstract}
This paper studies large dam models where the difference between
lower and upper levels, $L$, is assumed to be large. Passage across
the levels leads to damage, and the damage costs of crossing the
lower or upper level are proportional to the large parameter $L$.
Input stream of water is described by compound Poisson process, and
the water cost depends upon current level of water in the dam. The
aim of the paper is to choose the parameters of output stream
(specifically defined in the paper) minimizing the long-run
expenses. The particular problem, where input stream is ordinary
Poisson and water costs are not taken into account, has been studied
in [Abramov, \emph{J. Appl. Prob.}, 44 (2007), 249-258]. The present
paper addresses the question \textit{How does the structure of water
costs affect the optimal solution?} Under natural assumptions we
prove an existence and uniqueness of a solution and study the case
of linear structure of the costs.
\end{abstract}
\maketitle

\section{Introduction}\label{Introduction}
A large dam is defined by the parameters $L^{\mathrm{lower}}$ and
$L^{\mathrm{upper}}$, which are, respectively, the lower and upper
levels of the dam. If the current level is between these bounds, the
dam is assumed to be in a normal state. The difference $L =
L^{\mathrm{upper}}-L^{\mathrm{lower}}$ is large, and this is the
reason for calling the dam \textit{large}. This property enables us
to use asymptotic analysis as $L\to\infty$ and solve different
problems of optimal control, which by a direct way, that is without
using an asymptotic analysis, become very hard.

Let $L_t$ denote the water level at time $t$. If
$L^{\mathrm{lower}}<L_t\leq L^{\mathrm{upper}}$, then the state of
the dam is called \textit{normal}. Passage across lower or upper
level leads to damage. The costs per time unit of this damage are
$J_1=j_1L$ for the lower level and, respectively, $J_2=j_2L$ for the
upper level, where $j_1$ and $j_2$ are given real constants. The
water inflow is described by a compound Poisson process. Namely, the
probability generating function of input amount of water (which is
assumed to be an integer-valued random variable) in an interval $t$
is given by
 \begin{equation}\label{pgfArrival}
 f_t(z)=\exp\left\{-\lambda t\left(1-\sum_{i=1}^{\infty}r_iz^i\right)\right\},
 \end{equation}
where $r_i$ is the probability that at a specified moment of Poisson
arrival the amount of water will increase by $i$ units. In practice
this means that the arrival of water is registered at random
instants $t_1$, $t_2$, \ldots; the times between consecutive
instants are mutually independent and exponentially distributed with
parameter $\lambda$, and quantities of water (number of water units)
of input flow are specified as a quantity $i$ with probability $r_i$
($r_1+r_2+\ldots=1$). Clearly that this assumption is more
applicable to real world problems than the assumption of
\cite{Abramov 2007} where the inter-arrival times of water units are
exponentially distributed with parameter $\lambda$. For example, the
assumption made in the present paper enables us to approach a
continuous dam model, assuming that the water levels $L_t$ take the
discrete values $\{j\Delta\}$, where $j$ is a positive integer and
step $\Delta$ is a positive small real constant. In the paper,
however, the water levels $L_t$ are assumed to be integer-valued.

The outflow of water is state-dependent as follows. If the level of
water is between $L^{\mathrm{lower}}$ and $L^{\mathrm{upper}}$, then
an interval between departures of water units (inverse output flow)
has the probability distribution function $B_1(x)$. If the level of
water exceeds $L^{\mathrm{upper}}$, then an inverse output flow has
the probability distribution function $B_2(x)$. The probability
distribution function $B_2(x)$ is assumed to obey the condition
$\int_0^\infty x\mbox{d}B_2(x)<\frac{1}{\lambda}$. If the level of
water is $L^{\mathrm{lower}}$ exactly, then output of water is
frozen, and it resumes again as soon as the level of water exceeds
the level $L^{\mathrm{lower}}$. (The exact mathematical formulation
of the problem taking into account some specific details is given
below.)

Let $c_{L_t}$ denote the cost of water at level $L_t$. The
sequence $c_i$ is assumed to be positive and non-increasing. The
problem of the present paper is to choose the parameter
$\int_0^\infty x\mbox{d}B_1(x)$ of the dam in the normal state
minimizing the objective function
\begin{equation}\label{I1}
J = p_1J_1 + p_2J_2 + \sum_{i=L^{\mathrm{lower}}+1}^{L^{\mathrm{upper}}}c_iq_i,
\end{equation}
where
\begin{eqnarray}
p_1&=&\lim_{t\to\infty}\mathsf{Pr}\{L_t=L^{\mathrm{lower}}\},\label{I2}\\
p_2&=&\lim_{t\to\infty}\mathsf{Pr}\{L_t>L^{\mathrm{upper}}\},\label{I3}\\
q_i&=&\lim_{t\to\infty}\mathsf{Pr}\{L_t=L^{\mathrm{lower}}+i\}, \
i=1,2,\ldots,L\label{I4}.
\end{eqnarray}
Usually the level $L^{\mathrm{lower}}$ is identified with an empty
queue (i.e. $L^{\mathrm{lower}}:=0$ and $L^{\mathrm{upper}}:=L$),
and the dam model is the following queueing system with service
depending on queue-length. If immediately before a service beginning
the queue-length exceeds the level $L$, then the customer is served
by the probability distribution function $B_2(x)$. Otherwise, the
service time distribution is $B_1(x)$. The value $p_1$ is the
stationary probability of an empty system, the value $p_2$ is the
stationary probability that a customer is served by probability
distribution $B_2(x)$, and $q_i$, $i=1,2,\ldots,L$, are the
stationary probabilities of the queue-length process, so
$p_1+p_2+\sum\limits_{i=1}^L q_i=1$. (For the described queueing
system, the right-hand side limits in relations
\eqref{I2}-\eqref{I4} do exist.)

In our study, the parameter $L$ increases unboundedly, and we deal
with the series of queueing systems. The above parameters, such as
$p_1$, $p_2$, $J_1$, $J_1$ as well as other parameters are functions
of $L$. The argument $L$ will be often omitted in these functions.

Similarly to \cite{Abramov 2007}, it is assumed that the input
parameter $\lambda$, the probabilities $r_1$, $r_2$,\ldots and
probability distribution function $B_2(x)$ are given, while the
appropriate probability function $B_1(x)$ should be chosen from the
specified parametric family of functions $B_1(x,C)$. (Actually, we
deal with the family of probability distribution functions $B_1(x)$
depending on two parameters $\delta$ and $L$ in series, i.e.
$B_1(x,\delta,L)$. Then the parametric family of distributions
$B_1(x,C)$ is described in the limiting scheme as $\delta L\to C$,
so the parameter $C$ belongs to the family of possible limits of
$\delta L$ as $\delta\to0$ and $L\to\infty$.)

 The outflow rate,
should be chosen such that to minimize the objective function of
\eqref{I1} with respect to the parameter $C$, which results in
choice of the corresponding probability distribution function
$B_1(x,C)$ of that family.

A particular problem have been studied in \cite{Abramov 2007}. A
circle of problems associated with the results of \cite{Abramov
2007} are discussed in a review paper \cite{Abramov 2009}.

The simplest model with Poisson input stream and the objective
function having the form $J= p_1J_1 + p_2J_2$ (i.e. the water costs
are not taken into account), has been studied in \cite{Abramov
2007}. Denote $\rho_2=\lambda\int_0^\infty x\mbox{d}B_2(x)$ and
$\rho_1=\rho_1(C)=\lambda\int_0^\infty x\mbox{d}B_1(x,C)$. (The
parameter $C$ is a unique solution of a specific minimization
problem precisely formulated in \cite{Abramov 2007}.) In was shown
in \cite{Abramov 2007} that the solution to the control problem is
unique and has one of the following three forms:

(i) in the case $j_1=j_2\frac{\rho_2}{1-\rho_2}$ the optimal
solution is $\rho_1=1$;

(ii) in the case $j_1>j_2\frac{\rho_2}{1-\rho_2}$ the optimal
solution has the form $\rho_1=1+\delta$, where $\delta(L)$ is a
small positive parameter, and $\delta(L)L\to C$ as $L\to\infty$;

(iii) in the case $j_1<j_2\frac{\rho_2}{1-\rho_2}$, the optimal
strategy has the form $\rho_1=1-\delta$, and $\delta(L)L\to C$ as
$L\to\infty$.

It has been also shown in \cite{Abramov 2007} that the solution to
the control problem is insensitive to the type of probability
distributions $B_1(x)$ and $B_2(x)$. Specifically, it is expressed
via the first moment of $B_2(x)$ and the first two moments of
$B_1(x)$.

The aforementioned cases (i), (ii) and (iii) fall into the category
of heavy traffic analysis in queueing theory. There are many papers
related to this subject. We mention the books of Chen and Yao
\cite{Chen and Yao 2001} and Whitt \cite{Whitt 2001}, where a reader
can find many other references. The aforementioned paper
\cite{Abramov 2007} as well as the present paper, however, are
conceptually close to the well-known paper of Halfin and Whitt
\cite{Halfin and Whitt 1981}. The heavy-traffic conditions in
queueing systems with large number of identical servers arise
naturally if we assume that a high-level, associated with the loss
in that system, is reached with a given positive probability, while
the traffic intensities converge to 1 from the below and the arrival
rates and number of servers increase to infinity. In the case of the
single-server state-dependent queueing systems of the present paper
that model a large dam, we assume that the specified costs for
reaching the lower and upper levels multiplied by the corresponding
probabilities must converge to the given fixed values in limit, and
the resulting functional containing these quantities must be
minimized. This leads to the study of the family of systems under
the heavy-traffic behaviour, in which the sequence of products
$\delta L$ must converge to the optimal value $C$.

Compared to the earlier studies in \cite{Abramov 2007}, the solution
of the problems in the present paper requires a much deepen and
delicate analysis. The results of \cite{Abramov 2007} are extended
in two directions: (1) the arrival process is compound Poisson
rather than Poisson, and (2) structure of water costs in dependence
of the level of water in the dam is included.

The first extension leads to new techniques of stochastic analysis.
The main challenge in \cite{Abramov 2007} was to reduce the certain
characteristics of the system during a busy period to the
convolution type recurrence relation such as
$Q_n=\sum\limits_{i=0}^n Q_{n-i+1}f_i$ ($Q_0\neq0$), where $f_0>0$,
$f_i\geq0$ for all $i\geq1$, $\sum\limits_{i=0}^\infty f_i=1$ and
then to use the known results on the asymptotic behaviour of $Q_n$
as $n\to\infty$. In the case when arrivals are compound Poisson, the
same characteristics of the system cannot be reduced to the
aforementioned convolution type of recurrence relation. Instead, we
obtain a more general scheme including as a part the aforementioned
recurrence relation. In this case, asymptotic analysis of the
required characteristics becomes very challenging. It is based on
special stochastic domination methods, which will be explained in
details later.

The second extension leads to new analytic techniques of asymptotic
analysis. Asymptotic methods of \cite{Abramov 2007} do not longer
work, and one should use more delicate techniques instead. That is,
instead of Tak\'acs' asymptotic theorems \cite{Takacs 1967}, p.
22-23, one should use special Tauberian theorems with remainder by
Postnikov \cite{Postnikov 1980}, Sect. 25. For different
applications of the aforementioned Tak\'acs' asymptotic theorems and
Tauberian theorems of Postnikov see \cite{Abramov 2009}.

Another challenging problem for the dam model in the present paper
is the solution to the control problem, that is, the proof of a
uniqueness of the optimal solution. In the case of the model in
\cite{Abramov 2007} the existence and uniqueness of a solution
follows automatically from the explicit representations of the
functionals obtained there. (The existence of a solution follows
from the fact that in the case $\rho_1=1$ we get a bounded value of
the functional, while in the cases $\rho_1<1$ and $\rho_1>1$ the
functional is unbounded. Then the uniqueness of a solution reduces
to elementary minimization problem for smooth convex functions.) In
the case of the model in the present paper, the solution of the
present problem with extended criteria \eqref{I1} is related to the
same class of solutions as in \cite{Abramov 2007}. That is, it must
be either $\rho_1=1$ or one of two limits of $\rho_1=1+\delta$,
$\rho_1=1-\delta$ for positive small vanishing $\delta$ as $L$
increases unboundedly, and $L\delta\to C$. While the existence of a
solution follows trivially as in \cite{Abramov 2007}, the proof of a
uniqueness of the solution requires elegant techniques of the theory
of analytic functions and majorization inequalities (see \cite{Hardy
Littlewood Polya 1952} and \cite{Marschall Olkin 1979}).


Similarly to \cite{Abramov 2007}, we use the notation
$\rho_{1,l}=\lambda^l\int_0^\infty x^l\mbox{d}B_1(x)$, $l=2,3$. The
existence of $\rho_{1,l}$ (i.e. the moments of the third order of
$B_1(x)$) will be specially assumed in the formulations of the
statements corresponding to case studies.

It is assumed in the present paper that $c_i$ is a non-increasing
sequence. If the cost sequence $c_i$ were an arbitrary bounded
sequence, then a richer class of possible cases could be studied.
However, in the case of arbitrary cost sequence, the solution need
not be unique, and arbitrary costs $c_i$, say increasing in $i$,
seem not to be useful and, therefore, are not considered here. A
non-increasing sequence $c_i$ depends on $L$ in series. This means
that as $L$ changes (increasing to infinity) we have different
non-increasing sequences (see example in Section \ref{Examples}).
The initial value $c_1$ and final value $c_L$ are taken fixed and
strictly positive, and the limit of $c_L$ as $L\to\infty$ is assumed
to be positive as well.

Realistic models arising in practice assume that the probability
distribution function $B_1(x)$ should also depend on $i$, i.e have
the representation $B_{1,i}(x)$. The model of the present paper,
where $B_1(x)$ is the same for all $i$, under appropriate additional
information can approximate those more general models. Namely, one
can suppose that the stationary service time distribution $B_1(x)$
has the representation $B_1(x)=\sum\limits_{i=1}^Lq_iB_{1,i}(x)$
($q_i$, $i=1,2,\ldots,L$ are the state probabilities), and the
solution to the control problem for $B_1(x)$ enables us to find then
the approximate solutions to the control problem for $B_{1,i}(x)$,
$i=1,2,\ldots,L$ by using the Bayes rule. For example, the simplest
model can be of the form $B_1(x)=aB_{1}^*(x)+bB_1^{**}(x)$, where
$a:=\sum\limits_{i=1}^{L^0}q_i$ ($L^0<L$), and, respectively,
$b:=\sum\limits_{i=L^0+1}^{L}q_i$.

\smallskip
In the present paper we address the following questions.
\smallskip

$\bullet$ Uniqueness of an optimal solution and its structure.

$\bullet$ Interrelation between the parameters $j_1$, $j_2$,
$\rho_2$, $c_i$ ($i=1,2,\ldots,L$) when the optimal solution is
$\rho_1=1$.

\smallskip
The uniqueness of an optimal solution is given by Theorem
\ref{thm3}. In the case of the model considered in \cite{Abramov
2007} the condition for $\rho_1=1$ is
$j_1=j_2\frac{\rho_2}{1-\rho_2}$. Intuitive explanation of this
result is based on the well-known property of the stream of losses
during a busy period of $M/GI/1/n$ queues, under the assumption that
the expected interarrival and service times are equal (see Abramov
\cite{Abramov 1997}, Righter \cite{Righter 1999} and Wolff
\cite{Wolff 2002}). In the case of the model in this paper the
interrelation between the aforementioned and some additional
parameters involves the inequality (see Section \ref{Solution},
Corollary \ref{cor2}). Exact results are obtained in the particular
case of linearly decreasing costs as the level of water increases
(for brevity, this case is called \textit{linear costs}). In this
case, a numerical solution of the problem is given.
\smallskip

The rest of the paper is organized as follows. In Section
\ref{Methodology} the main ideas and methods of asymptotic analysis
are given. In Section \ref{MG1}, we recall the basic methods related
to state dependent queueing system with ordinary Poisson input that
have been used in \cite{Abramov 2007}. Then in Section \ref{MXG1},
extensions of these methods for the model considered in this paper
are given. Specifically, the methodology of constructing linear
representations between mean characteristics given during a busy
period is explained. In Section \ref{Stationary probabilities}, the
asymptotic behavior of the stationary probabilities is studied. In
Section \ref{Preliminaries}, known Tauberian theorems that are used
in the asymptotic analysis in the paper are recalled. In Section
\ref{Preliminary asymp}, exact formulae for the stationary
probabilities $p_1$ and $p_2$ are derived. On the basis of these
formulae, in Sections \ref{Final asymp} and \ref{Final asymp 2} the
asymptotic theorems for the stationary probabilities $p_1$ and $p_2$
have been established. Section \ref{Q stationary probabilities} is
devoted to asymptotic analysis of the stationary probabilities
$q_{L-i}$, $i=1,2,\ldots$. In Section \ref{Explicit q}, the explicit
representation for the stationary probabilities $q_i$ is derived. On
the basis of this explicit representation and Tauberian theorems, in
following Sections \ref{Case 1}, \ref{Case 2} and \ref{Case 3}
asymptotic theorems for these stationary probabilities are
established in the cases $\rho_1=1$, $\rho_1=1+\delta$ and
$\rho_1=1-\delta$ correspondingly, where positive $\delta$ is
assumed to vanish such that $\delta L\to C$ as $L\to\infty$. In
Section \ref{ObFunction} the objective function given in \eqref{I1}
is studied. In following Sections \ref{Case 1O}, \ref{Case 2O} and
\ref{Case 3O}, the asymptotic theorems for this objective function
are established for the cases $\rho_1=1$, $\rho_1=1+\delta$ and
$\rho_1=1-\delta$ correspondingly. In Section \ref{Solution}, the
theorem on existence and uniqueness of a solution is proved. In
Section \ref{Examples}, the case of linear costs is studied and
relevant numerical results are provided.

\section{Methodology of analysis}\label{Methodology}

In this section we describe the methodology used in the present
paper. This is a very important step because the earlier methods of
\cite{Abramov 2007} do not work for this extended model and hence
need in substantial revision.

We start from the model where arrivals are Poisson, and then we
explain how the methods should be developed for the model where an
arrival process is compound Poisson.

\subsection{State dependent queueing system with Poisson input and its characteristics}\label{MG1}

In this section we consider the simplest model in which arrival flow
is Poisson with parameter $\lambda$. Let $T_L$ denote the length of
a busy period of this system, and let $T_L^{(1)}$, $T_L^{(2)}$
denote the cumulative times spent for service of customers arrived
during that busy period with probability distribution functions
$B_1(x)$ and $B_2(x)$ correspondingly. For $k=1,2$, the expectations
of service times will be denoted by $\frac{1}{\mu_k}=\int_0^\infty
x\mathrm{d}B_k(x)$, and the loads by $\rho_k=\frac{\lambda}{\mu_k}$.
Let $\nu_L$, $\nu_L^{(1)}$ and $\nu_L^{(2)}$ denote correspondingly
the number of served customers during a busy period, and the numbers
of those customers served by the probability distribution functions
$B_1(x)$ and $B_2(x)$. The random variable $T_L^{(1)}$ coincides in
distribution with a busy period of the $M/GI/1/L$ queueing system
($L$ is the number of waiting places excluding the place for
server). The elementary explanation of this fact is based on a
property of level crossings and the property of the lack of memory
of exponential distribution (e.g. \cite{Abramov 2007}), so the
analytic representation for $\mathsf{E}T_L^{(1)}$ is the same as
this for the expected busy period of the $M/GI/1/L$ queueing system.
The recurrence relation for the Laplace-Stieltjes transform and
consequently that for the expected busy period of the $M/GI/1/L$
queueing system has been derived by Tomko \cite{Tomko}. So, for
$\mathsf{E}T_L^{(1)}$ the following recurrence relation is
satisfied:
\begin{equation}\label{MG1.1}
\mathsf{E}T_L^{(1)}=\sum_{i=0}^L\mathsf{E}T_{L-i+1}^{(1)}\int_0^\infty\mathrm{e}^{-\lambda
x}\frac{(\lambda x)^i}{i!} \mathrm{d}B_1(x),
\end{equation}
where $\mathsf{E}T_0^{(1)}=\frac{1}{\mu_1}$. (The random variable
$T_i^{(1)}$ is defined similarly to that of $T_L^{(1)}$. The only
difference is in the state parameter $i$ that is given instead of
$L$.) Recurrence relation \eqref{MG1.1} is a particular form of the
recurrence relation
\begin{equation}\label{MG1.1.1}
Q_n=\sum_{i=0}^n Q_{n-i+1}f_i,
\end{equation}
where $Q_0\neq0$, $f_0>0$, $f_i\geq0$, $i=1,2,\ldots$ and
$\sum\limits_{i=0}^\infty f_i=1$ (see Tak\'acs \cite{Takacs 1967}).

Using the obvious system of equations:
\begin{eqnarray}
\mathsf{E}T_L&=&\mathsf{E}T_{L}^{(1)}+\mathsf{E}T_{L}^{(2)},\label{MG1.2}\\
\mathsf{E}\nu_L&=&\mathsf{E}\nu_{L}^{(1)}+\mathsf{E}\nu_{L}^{(2)},\label{MG1.3}
\end{eqnarray}
and Wald's equations (see \cite{Feller 1966}, p.384)
\begin{eqnarray}
\mathsf{E}T_L^{(1)}&=&\frac{1}{\mu_1}\mathsf{E}\nu_L^{(1)},\label{MG1.4}\\
\mathsf{E}T_L^{(2)}&=&\frac{1}{\mu_2}\mathsf{E}\nu_L^{(2)},\label{MG1.5}
\end{eqnarray}
one can express the quantities $\mathsf{E}T_L$, $\mathsf{E}\nu_L$,
$\mathsf{E}T_L^{(2)}$, $\mathsf{E}\nu_L^{(1)}$ and
$\mathsf{E}\nu_L^{(2)}$ all via $\mathsf{E}T_{L}^{(1)}$ as the
linear functions. Indeed, taking into account that the number of
arrivals during a busy cycle coincides with the total number of
customers served during a busy period we have
\begin{equation}\label{MG1.5.1}
\lambda\mathsf{E}T_L+1=\mathsf{E}\nu_L,
\end{equation}
which together with
\eqref{MG1.2}-\eqref{MG1.5} yields the linear representations
required.

For example,
\begin{equation}\label{MG1.5.2}
\mathsf{E}\nu_L^{(2)}=\frac{1}{1-\rho_2}-\frac{1}{\mu_1}\cdot\frac{1-\rho_1}{1-\rho_2}\mathsf{E}T_L^{(1)},
\end{equation}
and
\begin{equation}\label{MG1.5.3}
\mathsf{E}T_L^{(2)}=\frac{\rho_2}{\lambda(1-\rho_2)}-\frac{\rho_2}{\lambda}
\cdot\frac{1-\rho_1}{1-\rho_2}\mathsf{E}T_L^{(1)}.
\end{equation}
As a result, the stationary probabilities $p_1$ and $p_2$ both are
expressed via $\mathsf{E}\nu_{L}^{(1)}$ as follows:
\begin{equation*}
p_1=\frac{1-\rho_2}{1+(\rho_1-\rho_2)\mathsf{E}\nu_L^{(1)}},
\end{equation*}
\begin{equation*}
p_2=\frac{\rho_2+\rho_2(\rho_1-1)\mathsf{E}\nu_L^{(1)}}{1+(\rho_1-\rho_2)\mathsf{E}\nu_L^{(1)}}
\end{equation*}
(see Section 2 of \cite{Abramov 2007} for further details). It is
interesting to note that the coefficients in linear representation
all are insensitive to the probability distribution functions
$B_1(x)$ and $B_2(x)$ and are only expressed via parameters such as
$\mu_1$, $\mu_2$ and $\lambda$.

The asymptotic behaviour of $\mathsf{E}T_L^{(1)}$ as $L\to\infty$
that given by \eqref{MG1.1} is established on the basis of the known
asymptotic behaviour of the sequence $Q_n$ as $n\to\infty$ that
given by \eqref{MG1.1.1} (see \cite{Takacs 1967}, p.22,
\cite{Postnikov 1980} as well as recent paper \cite{Abramov 2009}).
To make the paper self-contained, the necessary results about the
asymptotic behaviour of $Q_n$ as $n\to\infty$ are given in Section
\ref{Preliminaries}.

\subsection{State dependent queueing system with compound Poisson input and its characteristics}\label{MXG1}

For $M^X/GI/1/L$ queues, certain characteristics associated with
busy periods have been studied by Rosenlund \cite{Rosenlund 1973}.
Developing the results of Tomko \cite{Tomko}, Rosenlund
\cite{Rosenlund 1973} has derived the recurrence relations for the
joint Laplace-Stieltjes and $z$-transform of two-dimensional
distributions of a generalized busy period and the number of
customers served during that period. In turn, both of these
approaches \cite{Tomko} and \cite{Rosenlund 1973} are based on a
well-known Tak\'acs' method (see \cite{Takacs 1955} or \cite{Takacs
1962}).

For further analysis, \cite{Rosenlund 1973} used matrix-analytic
techniques and techniques of the theory of analytic functions. This
type of analysis is very hard and seems cannot be easily adapted for
the purposes of the present paper, where a more general model than
that from a paper \cite{Rosenlund 1973} is studied.

In this section we explain how the method of Section \ref{MG1} can be extended, and how the characteristics of the system can be expressed via the similar convolution type recurrence relations.

Notice first, that the linear representations similar to those
derived for the state dependent queueing system with ordinary
Poisson input are satisfied for the present system as well. Indeed,
equations \eqref{MG1.2}-\eqref{MG1.5} all hold in the case of the
present queueing system. The only difference is that instead of
\eqref{MG1.5.1}, the relation between $\mathsf{E}T_L$ and
$\mathsf{E}\nu_L$ should be
\begin{equation}\label{MXG1.0}
\lambda\mathsf{E}\varsigma\mathsf{E}T_L+\mathsf{E}\varsigma=\mathsf{E}\nu_L,
\end{equation}
where $\varsigma$ denotes a batch size of an arrival. (The random
variable $\varsigma$ has the distribution
$\mathsf{Pr}\{\varsigma=i\}=r_i$.) This leads to a slight change of
the linear representations mentioned in Section \ref{MG1}. The main
difficulty, however, is that the recurrence relation for
$\mathsf{E}T_L^{(1)}$ (or the corresponding quantity
$\mathsf{E}\nu_L^{(1)}$) is no longer a convolution type recurrence
relation as \eqref{MG1.1.1}. So, we should use another type of
analysis, which is explained below.

For this model, let $\widetilde T_j$, $j=1,2,\ldots, L$, denote the
time interval starting from the moment when there are $L-j+1$
customers in the system until the moment when there remain $L-j$
customers for the first time since its beginning. Similarly to the
notation used in Section \ref{MG1}, let us introduce the random
variables $\widetilde{T}_j^{(1)}$, $\widetilde{T}_j^{(2)}$,
$\widetilde \nu_j$, $\widetilde{\nu}_j^{(1)}$,
$\widetilde{\nu}_j^{(2)}$, $j=1,2,\ldots, L$, which have the same
meaning as before. Specifically, when $j$ takes the value $L$,
$\widetilde T_L$ is the length a busy period starting from a single
customer (1-busy period); $\widetilde\nu_L$ is the number of
customers that served during a 1-busy period, and so on.

With the aid of the aforementioned Tak\'{a}cs' method \cite{Takacs
1955}, \cite{Takacs 1962}, one can derive the recurrence relation
similar to that of \eqref{MG1.1}. Namely,
\begin{equation}\label{MXG1.1}
\mathsf{E}\widetilde{T}_L^{(1)}=\sum_{i=0}^{L}\mathsf{E}\widetilde{T}_{L-i+1}^{(1)}\int_0^\infty
\frac{1}{i!}\frac{\mathrm{d}^{i}f_x(z)}{\mathrm{d}z^i}\Big|_{z=0}
\mathrm{d}B_1(x),
\end{equation}
where $\mathsf{E}\widetilde{T}_0^{(1)}=\frac{1}{\mu_1}$, and the
generating function $f_x(z)$ is given by \eqref{pgfArrival}. So, the
only difference between \eqref{MG1.1} and \eqref{MXG1.1} is in their
integrands, and in particular case $r_1=1$, $r_i=0$, $i\geq2$ we
clearly arrive at the same expressions. The explicit results
associated with recurrence relation \eqref{MXG1.1} is given later in
the paper. Apparently, the similar system of equations as
\eqref{MG1.1} - \eqref{MG1.5} is satisfied for the characteristics
of the state dependent queueing system $M^X/GI/1$. Namely,
\begin{eqnarray}
\mathsf{E}\widetilde T_L&=&\mathsf{E}\widetilde T_{L}^{(1)}+\mathsf{E}\widetilde T_{L}^{(2)},\label{MXG1.2}\\
\mathsf{E}\widetilde\nu_L&=&\mathsf{E}\widetilde\nu_{L}^{(1)}+\mathsf{E}\widetilde\nu_{L}^{(2)},\label{MXG1.3}
\end{eqnarray}
\begin{eqnarray}
\mathsf{E}\widetilde T_L^{(1)}&=&\frac{1}{\mu_1}\mathsf{E}\widetilde\nu_L^{(1)},\label{MXG1.4}\\
\mathsf{E}\widetilde
T_L^{(2)}&=&\frac{1}{\mu_2}\mathsf{E}\widetilde\nu_L^{(2)}.\label{MXG1.5}
\end{eqnarray}
Therefore, the same linear representations via
$\mathsf{E}\widetilde{T}_L^{(1)}$ hold for characteristics of these
systems, where by $\rho_1$ and $\rho_2$ one now should mean the
expected numbers of \textit{arrived customers} per service time (not
the expected number of arrivals) having the probability distribution
function $B_1(x)$ and, respectively, $B_2(x)$.

Let us now consider the length of a busy period $T_L$ and associated
random variables $T_{L}^{(1)}$, $T_{L}^{(2)}$, $\nu_L$,
$\nu_{L}^{(1)}$ and $\nu_{L}^{(2)}$. Let $\varsigma_1$ denote a size
of batch that starts a busy period. (An integer random variable
$\varsigma_1$ has the distribution
$\mathsf{Pr}\{\varsigma=i\}=r_i$.) Then $T_L$ can be represented
\begin{equation}\label{MXG1.6}
T_L{\buildrel d\over =}\sum_{i=1}^{\varsigma_1\wedge
(L+1)}\widetilde{T}_{L-i+1}+\sum_{i=1}^{\varsigma_1-(L+1)}\widetilde{T}_{0,i},
\end{equation}
where 1-busy periods $\widetilde{T}_{L-i+1}$, $i=1,2,\ldots,L$ are mutually independent;

\smallskip
$\widetilde{T}_{0}$ denotes a special 1-busy period that starts from
a service time having the probability distribution function $B_1(x)$
and all other service times are mutually independent and identically
distributed random variables having the probability distribution
$B_2(x)$, and the distributions of interarrival times and batch
sizes are the same as in the original state dependent queueing
system;

\smallskip
$\widetilde{T}_{0,i}$, $i=1,2,\ldots$, is a sequence of independent and identically distributed 1-busy periods of the $M^X/G/1$ queueing system, the service times of which all are independent and identically distributed random variables having the probability distribution function $B_2(x)$, and the distributions of interarrival times and batch sizes are the same as in the original state dependent queueing system;

\smallskip
$a\wedge b$ denotes $\min\{a,b\}$;

\smallskip
${\buildrel d\over =}$ denotes the equality in distribution;

\smallskip
in the case where $\varsigma_1-(L+1)\leq0$, the empty sum in
\eqref{MXG1.6} is assumed to be zero.

\medskip
In turn, the representation for $T_L^{(1)}$ is as follows:
\begin{equation}\label{MXG1.7}
T_L^{(1)}{\buildrel d\over=}\sum_{i=1}^{\varsigma_1\wedge
(L+1)}\widetilde{T}_{L-i+1}^{(1)},
\end{equation}
where $\widetilde{T}_0^{(1)}$ denotes a single service time having
the probability distribution function $B_1(x)$. Whereas
$\mathsf{E}\widetilde{T}_L^{(1)}$ is determined by recurrence
relation \eqref{MXG1.1}, which is a particular case of
\eqref{MG1.1.1}, a convolution type recurrence relation is no longer
valid for $\mathsf{E}T_L^{(1)}$.

For the following asymptotic analysis of $\mathsf{E}T_L^{(1)}$ and
other mean characteristics such as $\mathsf{E}\nu_L^{(1)}$,
$\mathsf{E}\nu_L^{(2)}$ we will use the following techniques. We
first study the mean characteristics
$\mathsf{EE}\{T_{L}|\varsigma_1\wedge L\}$,
$\mathsf{EE}\{T_{L}^{(1)}|\varsigma_1\wedge L\}$,
$\mathsf{EE}\{T_{L}^{(2)}|\varsigma_1\wedge L\}$,
$\mathsf{EE}\{\nu_{L}^{(1)}|\varsigma_1\wedge L\}$ and
$\mathsf{EE}\{\nu_{L}^{(2)}|\varsigma_1\wedge L\}$. Then, assuming
that $L\to\infty$, we have
$\lim\limits_{L\to\infty}\mathsf{Pr}\{\varsigma_1\wedge
L=i\}=\mathsf{Pr}\{\varsigma_1=i\}$, as well as
$\lim\limits_{L\to\infty}\mathsf{E}\mathsf{E}\{T_L|\varsigma_1\wedge
L\}= \lim\limits_{L\to\infty}\mathsf{E}T_L$, and the similar limits
hold for the other mean characteristics. The following asymptotic
behaviour of the probabilities $p_1$ and $p_2$ as $L\to\infty$ is
then established similarly to that in \cite{Abramov 2007}.

Let us show the justice of the linear representations that similar
to those \eqref{MG1.5.2} and \eqref{MG1.5.3}. Write
$\mathsf{E}\widetilde{T}_L=a+b\mathsf{E}\widetilde T_L^{(1)}$, where
$a$ and $b$ are specified constants. Then, by the total probability
formula,
\begin{equation}\label{MXG1.8}
\begin{aligned}
\mathsf{E}\mathsf{E}\{T_L|\varsigma_1\wedge L\}
&=\sum_{i=1}^L\mathsf{Pr}\{\varsigma_1\wedge L=i\}\sum_{j=1}^{i}\mathsf{E}\widetilde T_{L-j+1}\\
&=\sum_{i=1}^L\mathsf{Pr}\{\varsigma_1\wedge L=i\}\sum_{j=1}^i(a+b\mathsf{E}\widetilde T_{L-i+1}^{(1)})\\
&=a\sum_{i=1}^L i\mathsf{Pr}\{\varsigma_1\wedge L=i\}+b\sum_{i=1}^L
\mathsf{Pr}\{\varsigma_1\wedge L=i\}
\sum_{j=1}^i\mathsf{E}\widetilde T_{L-i+1}^{(1)}\\
&=a\mathsf{E}(\varsigma_1\wedge
L)+b\mathsf{E}\mathsf{E}\{T_{L}^{(1)}|\varsigma_1\wedge L\}.
\end{aligned}
\end{equation}
This representation is \textit{quazi-linear} in the sense that only
for $J\geq L$ (but not for all $J\geq1$)
\begin{equation*}
\mathsf{E}\mathsf{E}\{T_J|\varsigma_1\wedge
L\}=a\mathsf{E}(\varsigma_1\wedge
L)+b\mathsf{E}\mathsf{E}\{T_{J}^{(1)}|\varsigma_1\wedge L\}.
\end{equation*}
Apparently, the similar quazi-linear representations are satisfied
for the mean characteristics
$\mathsf{E}\mathsf{E}\{\nu_L|\varsigma_1\wedge L\}$,
$\mathsf{E}\mathsf{E}\{T_L^{(2)}|\varsigma_1\wedge L\}$, and
$\mathsf{E}\mathsf{E}\{\nu_L^{(2)}|\varsigma_1\wedge L\}$ all via
$\mathsf{E}\mathsf{E}\{T_{L}^{(1)}|\varsigma_1\wedge L\}$. The exact
values of the coefficients in these quazi-linear representations
will be derived in the next sections.

\section{Asymptotic theorems for the stationary probabilities $p_1$ and $p_2$}\label{Stationary
probabilities}

In this section, the explicit expressions are derived for the
stationary probabilities, and their asymptotic behavior is
studied. These results will be used in our further findings of the
optimal solution.

\subsection{Preliminaries}\label{Preliminaries}

In this section we recall the main properties of recurrence relation
\eqref{MG1.1.1}. The detailed theory of these recurrence relations
can be found in Tak\'{a}cs \cite{Takacs 1967}. For the generating
function $Q(z)=\sum\limits_{j=0}^\infty Q_jz^j$, $|z|\leq1$ we have
\begin{equation}\label{SP4+1}
Q(z)=\frac{Q_0 F(z)}{F(z)-z},
\end{equation}
where $F(z)=\sum\limits_{j=0}^\infty f_jz^j$.

Asymptotic behavior of $Q_n$ as $n\to\infty$ has been studied by
Tak\'{a}cs \cite{Takacs 1967} and Postnikov \cite{Postnikov 1980}.
Recall the theorems that are needed in this paper.

Denote $\gamma_m=\lim_{z\uparrow1}\frac{\mathrm{d}^mF(z)}{\mathrm{d}z^m}$.

\begin{lem}\label{lem.T} \texttt{(Tak\'acs \cite{Takacs 1967}, p.22-23).} If $\gamma_1<1$ then
\begin{equation}\label{T.1}
\lim_{n\to\infty} Q_n=\frac{Q_0}{1-\gamma_1}.
\end{equation}
If $\gamma_1=1$ and $\gamma_2<\infty$, then
$$
\lim_{n\to\infty}\frac{Q_n}{n}=\frac{2Q_0}{\gamma_2}.
$$
If $\gamma_1>1$, then
\begin{equation}\label{T.3}
\lim_{n\to\infty}\left(Q_n-\frac{Q_0}{\delta^n[1-F^\prime(\delta)]}\right)=\frac{Q_0}{1-\gamma_1},
\end{equation}
where $\delta$ is the least in absolute value root of the functional
equation $z=F(z)$.
\end{lem}

\begin{lem}\label{lem.P} \texttt{(Postnikov \cite{Postnikov 1980}, Sect.25).} Let $\gamma_1=1$, $\gamma_2<\infty$ and $f_0+f_1<1$. Then, as $n\to\infty$,
\begin{equation}\label{P.1}
Q_{n+1}-Q_n=\frac{2Q_0}{\gamma_2}+o(1).
\end{equation}
\end{lem}

\subsection{Exact formulae for $p_1$ and $p_2$}\label{Preliminary asymp}

In this section we derive exact representations for $p_1$ and $p_2$
via $\mathsf{E}\nu_L^{(1)}$. We also obtain some preliminary
asymptotic representations that easily follow from the explicit
results. Those asymptotic representations will be used in the
sequel.

We first start from the linear representations for
$\mathsf{E}\widetilde\nu_L^{(2)}$ in terms
$\mathsf{E}\widetilde\nu_L^{(1)}$. Namely, we have the following
lemma.

\begin{lem}\label{representation}
For $\mathrm{E}\widetilde{\nu}_L^{(2)}$, $L=1,2,\ldots$, we have the
following representation
\begin{equation}\label{SP4.0}
\mathsf{E}\widetilde{\nu}_L^{(2)}=\frac{1}{1-\rho_2}-\frac{1-\rho_1}{1-\rho_2}\mathsf{E}\widetilde\nu_L^{(1)},
\end{equation}
where $\rho_1=\frac{\lambda\mathsf{E}\varsigma}{\mu_1}$ and
$\rho_2=\frac{\lambda\mathsf{E}\varsigma}{\mu_2}<1$, and
$\mathsf{E}\widetilde\nu_L^{(1)}$ is given by
\begin{equation}\label{SP4.1}
\begin{aligned}
\mathsf{E}\widetilde{\nu}_L^{(1)}
&=\sum_{i=0}^{L}\mathsf{E}\widetilde{\nu}_{L-i+1}^{(1)}\int_0^\infty
\frac{1}{i!}\frac{\mathrm{d}^{i}f_x(z)}{\mathrm{d}z^i}\Big|_{z=0}
\mathrm{d}B_1(x),
\end{aligned}
\end{equation}
$\mathsf{E}\widetilde\nu_0^{(1)}=1$.
\end{lem}

\begin{proof}
Taking into account that the number of arrivals during 1-busy cycle
(1-busy period plus idle period) coincides with the number of
customers served during the same 1-busy period, according to Wald's
identity we have:
$$
\lambda\left(\mathsf{E}\widetilde{T}_L+\frac{1}{\lambda}\right)=\lambda\mathsf{E}\widetilde{T}_L+1
=\mathsf{E}\widetilde\nu_L=\mathsf{E}\widetilde\nu_L^{(1)}+\mathsf{E}\widetilde\nu_L^{(2)}.
$$
This equality together with \eqref{MXG1.2}-\eqref{MXG1.5} yields the
desired statement of the lemma, where \eqref{SP4.1} in turn follows
from \eqref{MXG1.1} and Wald's identity \eqref{MXG1.4}.
\end{proof}

The next step is to derive representations for
$\mathsf{EE}\{\nu_L^{(1)}|\varsigma_1\wedge L\}$ and
$\mathsf{EE}\{\nu_L^{(2)}|\varsigma_1\wedge L\}$. We have the
following lemma.
\begin{lem}\label{extended representation}
For $\mathsf{EE}\{\nu_L^{(2)}|\varsigma_1\wedge L\}$ we have
\begin{equation}\label{SP4.2}
\mathsf{EE}\{\nu_L^{(2)}|\varsigma_1\wedge
L\}=\frac{\mathsf{E}(\varsigma_1\wedge
L)}{1-\rho_2}-\frac{1-\rho_1}{1-\rho_2}\mathsf{EE}\{\nu_L^{(1)}|\varsigma_1\wedge
L\},
\end{equation}
where
\begin{equation*}\label{SP4.3}
\mathsf{EE}\{\nu_L^{(1)}|\varsigma_1\wedge
L\}=\sum_{i=1}^L\mathsf{Pr}\{\varsigma_1\wedge L=i\}
\sum_{j=1}^i\mathsf{E}\widetilde{\nu}_{L-j+1}^{(1)},
\end{equation*}
and $\mathsf{E}\widetilde{\nu}_{L-j+1}^{(1)}$, $j=1,2,\ldots,L$, are
given by \eqref{SP4.1}.
\end{lem}

\begin{proof} Following the same arguments as in \eqref{MXG1.8}, one can write
$$\mathsf{EE}\{\nu_L^{(2)}|\varsigma_1\wedge L\}
=a\mathsf{E}(\varsigma_1\wedge
L)+b\mathsf{EE}\{\nu_L^{(1)}|\varsigma_1\wedge L\}$$ for specified
constants $a$ and $b$ for which the linear representation
$\mathsf{E}\widetilde\nu_L^{(2)}=a+b\mathsf{E}\widetilde\nu_L^{(1)}$
is satisfied. Hence, according to relation \eqref{SP4.0} of Lemma
\ref{representation}, $a=\frac{1}{1-\rho_2}$ and
$b=-\frac{1-\rho_1}{1-\rho_2}$. The proof is completed.
\end{proof}

The following lemma yields exact estimates for the difference
$\mathsf{E}\nu_L^{(1)}-\mathsf{EE}\{\nu_L^{(1)}|\varsigma_1\wedge
L\}$.

\begin{lem}\label{estimate}
We have the following estimate:
\begin{equation}\label{SP4.4}
\mathsf{E}\nu_L^{(1)}-\mathsf{EE}\{\nu_L^{(1)}|\varsigma_1\wedge
L\}=\mathsf{Pr}\{\varsigma_1>L\},
\end{equation}
\end{lem}

\begin{proof} Similarly to \eqref{MXG1.7} we have
\begin{equation*}\label{SP4.6}
\nu_L^{(1)}{\buildrel
d\over=}\sum_{i=1}^{\varsigma_1\wedge(L+1)}\widetilde\nu_{L-i+1}^{(1)},
\end{equation*}
where $\widetilde\nu_{L-i+1}^{(1)}$, $i=1,2,\ldots,L$ are mutually
independent, and $\widetilde\nu_0^{(1)}=1$. Hence,
\begin{equation}\label{SP4.7}
\mathsf{E}\nu_L^{(1)}=\sum_{i=1}^{L+1}\mathsf{Pr}\{\varsigma_1\wedge(L+1)=i\}
\sum_{j=1}^i\mathsf{E}\widetilde\nu_{L-j+1}^{(1)}.
\end{equation}
In turn, the representation for
$\mathsf{EE}\{\nu_L^{(1)}|\varsigma_1\wedge L\}$ is
\begin{equation}\label{SP4.8}
\mathsf{EE}\{\nu_L^{(1)}|\varsigma_1\wedge
L\}=\sum_{i=1}^{L}\mathsf{Pr}\{\varsigma_1\wedge L=i\}
\sum_{j=1}^i\mathsf{E}\widetilde\nu_{L-j+1}^{(1)}.
\end{equation}
Subtracting \eqref{SP4.8} from \eqref{SP4.7} we obtain:
\begin{equation*}
\begin{aligned}
\mathsf{E}\nu_L^{(1)}-\mathsf{EE}\{\nu_L^{(1)}|\varsigma_1\wedge
L\}&=\mathsf{Pr}\{\varsigma_1=L\}\sum_{j=1}^L\mathsf{E}\widetilde\nu_j^{(1)}
+\mathsf{Pr}\{\varsigma_1>L\}\sum_{j=0}^L\mathsf{E}\widetilde\nu_j^{(1)}\\
&\ \ \ -\mathsf{Pr}\{\varsigma_1\geq L\}\sum_{j=1}^L\mathsf{E}\widetilde\nu_j^{(1)}\\
&=\mathsf{Pr}\{\varsigma_1>L\}.
\end{aligned}
\end{equation*}
Relation \eqref{SP4.4} is proved.
\end{proof}

From Lemma \ref{estimate} we have the following important corollary.

\begin{cor}\label{Asymp estimate} As $L\to\infty$,
\begin{equation}\label{SP4.17}
\mathsf{E}\nu_L^{(1)}-\mathsf{EE}\{\nu_L^{(1)}|\varsigma_1\wedge
L\}=o(1),
\end{equation}
and
\begin{equation}\label{SP4.18}
\mathsf{E}\nu_L^{(2)}-\mathsf{EE}\{\nu_L^{(2)}|\varsigma_1\wedge
L\}=o(1).
\end{equation}
\end{cor}

\begin{proof}
Asymptotic relation \eqref{SP4.17} follows immediately from
\eqref{SP4.4}. In order to show \eqref{SP4.18} let us first derive
the linear representation of $\mathsf{E}\nu_L^{(2)}$ via
$\mathsf{E}\nu_L^{(1)}$. From relation \eqref{MXG1.0} and equations
\eqref{MG1.2}-\eqref{MG1.5} in Section \ref{MG1}, which also hold
true in the case of the present queueing system with batch arrivals,
we obtain:
\begin{equation}\label{SP4.23}
\mathsf{E}\nu_L^{(2)}=\frac{\mathsf{E}\varsigma}{1-\rho_2}-\frac{1-\rho_1}{1-\rho_2}\mathsf{E}\nu_L^{(1)}.
\end{equation}
As well, for $\mathsf{EE}\{\nu_L^{(2)}|\varsigma_1\wedge L\}$ from
Lemma \ref{extended representation} we have representation
\eqref{SP4.2}. Hence, comparing the terms of \eqref{SP4.23} and
\eqref{SP4.2} and taking into account \eqref{SP4.17} we easily
arrive at asymptotic relation \eqref{SP4.18}. Lemma \ref{Asymp
estimate} is proved.
\end{proof}

The following lemma presents the exact formulae for the stationary
probabilities $p_1$ and $p_2$ via the term $\mathsf{E}\nu_L^{(1)}$.

\begin{lem}\label{Prel asymp}
We have:
\begin{equation}\label{SP4.19}
p_1=\frac{(1-\rho_2)\mathsf{E}\varsigma}{\mathsf{E}\varsigma+(\rho_1-\rho_2)\mathsf{E}\nu_L^{(1)}},
\end{equation}
and
\begin{equation}\label{SP4.20}
p_2=\frac{\rho_2\mathsf{E}\varsigma+\rho_2(\rho_1-1)\mathsf{E}\nu_L^{(1)}}
{\mathsf{E}\varsigma+(\rho_1-\rho_2)\mathsf{E}\nu_L^{(1)}}.
\end{equation}
\end{lem}

\begin{proof} Using renewal arguments (e.g. \cite{Ross 2000}) and relation \eqref{MXG1.0}, we have:
\begin{equation}\label{SP4.21}
p_1=\frac{\dfrac{1}{\lambda}}{\mathsf{E}T_L^{(1)}+\mathsf{E}T_L^{(2)}+\dfrac{1}{\lambda}}
=\frac{\mathsf{E}\varsigma}
{\mathsf{E}\nu_L^{(1)}+\mathsf{E}\nu_L^{(2)}}
\end{equation}
and
\begin{equation}\label{SP4.22}
p_2=\frac{\mathsf{E}T_L^{(2)}}{\mathsf{E}T_L^{(1)}+\mathsf{E}T_L^{(2)}+\dfrac{1}{\lambda}}=
\frac{\rho_2\mathsf{E}\nu_L^{(2)}}{\mathsf{E}\nu_L^{(1)}+\mathsf{E}\nu_L^{(2)}}.
\end{equation}


Now, substituting \eqref{SP4.23} for the right sides of \eqref{SP4.21} and \eqref{SP4.22} we obtain relations \eqref{SP4.19} and \eqref{SP4.20} of this lemma.
\end{proof}

\subsection{Asymptotic theorems for $p_1$ and $p_2$ under `usual assumptions'}\label{Final asymp}
By `usual assumption' we only mean the standard cases as $\rho_1=1$,
$\rho_1<1$ or $\rho_1>1$ for the asymptotic behaviour as
$L\to\infty$. In the following sections the heavy load assumptions
are assumed.

 The main result of Section \ref{Preliminary asymp} is Lemma
\ref{Prel asymp}, where the stationary probabilities $p_1$ and $p_2$
are expressed explicitly via $\mathsf{E}\nu_L^{(1)}$. The aim of
this section is to obtain the analogue of asymptotic Theorem 3.1 of
\cite{Abramov 2007}. To this end, we will derive an asymptotic
representation for $\mathsf{EE}\{\nu_L^{(1)}|\varsigma_1\wedge L\}$
as $L\to\infty$.

Let us first study the asymptotic behavior of
$\mathsf{E}\widetilde\nu_L^{(1)}$ as $L\to\infty$. For this purpose
derive the representation for the generating function
$\sum\limits_{j=0}^\infty\mathsf{E}\widetilde\nu_j^{(1)}z^j$. Using
representation \eqref{SP4.1}, we have:
\begin{equation}\label{SP5.0}
\begin{aligned}
\sum_{j=0}^\infty\mathsf{E}\widetilde\nu_j^{(1)}z^j&=\sum_{j=0}^\infty
z^j\sum_{i=0}^{j}
\mathsf{E}\widetilde{\nu}_{L-i+1}^{(1)}\int_0^\infty
\frac{1}{i!}\frac{\mathrm{d}^{i}f_x(u)}{\mathrm{d}u^i}\Big|_{u=0}
\mathrm{d}B_1(x)\\
&=\frac{U(z)}{U(z)-z},
\end{aligned}
\end{equation}
where
\begin{equation}\label{SP5.1}
\begin{aligned}
U(z)&=\int_0^{\infty}\exp\left\{-\lambda x\left(1-\sum_{i=1}^{\infty}r_iz^i\right)\right\}\mbox{d}B_1(x)\\
&=\widehat{B}_1(\lambda-\lambda\widehat{R}(z)).
\end{aligned}
\end{equation}
(By $\widehat{B}_1(s)$ we denote the Laplace-Stieltjes transform of
$B_1(x)$ ($\mathfrak{R}(s)\geq0$), and
$\widehat{R}(z)=\sum\limits_{i=1}^\infty r_iz^i$, $|z|\leq1$.)
Hence, from \eqref{SP5.1} and \eqref{SP5.0} we obtain:
\begin{equation}\label{SP5.2}
\sum_{j=0}^\infty\mathsf{E}\widetilde\nu_j^{(1)}z^j=\frac{\widehat{B}_1(\lambda-\lambda\widehat{R}(z))}
{\widehat{B}_1(\lambda-\lambda\widehat{R}(z))-z}.
\end{equation}

Notice, that the right-hand side of \eqref{SP5.0} and hence that of
\eqref{SP5.2} has the same form as \eqref{SP4+1}. Therefore,
according to Lemmas \ref{lem.T} and \ref{lem.P}, the asymptotic
behaviour of $\mathsf{E}\widetilde{\nu}_L^{(1)}$, as $L\to\infty$,
is given by the following statements.
\begin{lem}\label{Asymp behav 1}
If $\rho_1<1$, then
\begin{equation}\label{SP5.3}
\lim_{L\to\infty}\mathsf{E}\widetilde\nu_L^{(1)}=\frac{1}{1-\rho_1}.
\end{equation}
If $\rho_1=1$, and additionally $\rho_{1,2}=\int_0^\infty (\lambda
x)^2\mbox{d}B_1(x)<\infty$ and $\mathsf{E}\varsigma^2<\infty$, then
\begin{equation}\label{SP5.4}
 \mathsf{E}\widetilde\nu_L^{(1)}-\mathsf{E}\widetilde\nu_{L-1}^{(1)}
=\frac{2\mathsf{E}\varsigma}{\rho_{1,2}(\mathsf{E}\varsigma)^{3}+\mathsf{E}\varsigma^{2}-\mathsf{E}\varsigma
}+o(1).
\end{equation}
If $\rho_1>1$, then
\begin{equation}\label{SP5.5}
\lim_{L\to\infty}\left[\mathsf{E}\widetilde\nu_L^{(1)}-\frac{1}{\varphi^L[1+\lambda
\widehat{B}_1^\prime(\lambda-\lambda\widehat{R}(\varphi))\widehat{R}^\prime(\varphi)]}\right]=\frac{1}{1-\rho_1},
\end{equation}
where $\varphi$ is the root of the functional equation
$z=\widehat{B}_1(\lambda-\lambda\widehat{R}(z))$ that is least in
absolute value.
\end{lem}

\begin{proof} Asymptotic relations \eqref{SP5.3} and \eqref{SP5.5}
follow by application of those \eqref{T.1} and, respectively,
\eqref{T.3} of Lemma \ref{lem.T}.

In order to prove asymptotic relation \eqref{SP5.4} we should apply
the Tauberian theorem of Postnikov (Lemma \ref{lem.P}). Then
asymptotic relation \eqref{SP5.4} is to follow from \eqref{P.1} if
we prove that the Tauberian condition $f_0+f_1<1$ of Lemma
\ref{lem.P} is satisfied. In the case of the present model, we must
prove that for some $\lambda_0>0$ the equality
\begin{equation}\label{SP5.5+1}
\int_0^\infty\mathrm{e}^{-\lambda_0x}(1+\lambda_0r_1x)\mathrm{d}B_1(x)=1
\end{equation}
is not the case. Without loss of generality $r_1$ in \eqref{SP5.5+1}
can be set to be equal to 1, since
\begin{equation*}\label{SP5.5+0}
\int_0^\infty\mathrm{e}^{-\lambda_0x}(1+\lambda_0r_1x)\mathrm{d}B_1(x)\leq
\int_0^\infty\mathrm{e}^{-\lambda_0x}(1+\lambda_0x)\mathrm{d}B_1(x).
\end{equation*}
Thus, we should prove the inequality
\begin{equation*}\label{SP5.5+2}
\int_0^\infty\mathrm{e}^{-\lambda x}(1+\lambda x)\mathrm{d}B_1(x)<1.
\end{equation*}
Indeed, $\int_0^\infty\mathrm{e}^{-\lambda x}(1+\lambda
x)\mathrm{d}B_1(x)$ is an analytic function in $\lambda$, and hence,
according to the theorem on the maximum module of an analytic
function, equality \eqref{SP5.5+1} where $r_1=1$ must hold for
\textit{all} $\lambda_0\geq0$. This means that \eqref{SP5.5+1} is
valid if and only if
$$
\int_0^\infty\mathrm{e}^{-\lambda_0x}\frac{(\lambda_0
x)^i}{i!}\mathrm{d}B_1(x)=0
$$
for all $i\geq2$ and $\lambda_0\geq0$. Since
$\int_0^\infty\mathrm{e}^{-\lambda x}{(-x)^i}\mathrm{d}B_1(x)$ is
the $i$th derivative of the Laplace-Stieltjes transform $\widehat
B_1(\lambda)$, then in this case the Laplace-Stieltjes transform
$\widehat B_1(\lambda)$ must be a linear function in $\lambda$, i.e.
$\widehat B_1(\lambda)=d_0+d_1\lambda$, $d_0$ and $d_1$ are some
constants. However, since $|\widehat B_1(\lambda)|\leq1$, we have
$d_0=1$ and $d_1=0$. This is a trivial case where $B_1(x)$ is
concentrated in point 0, and therefore it is not a probability
distribution function having a positive mean.
 Thus \eqref{SP5.5+1} is not the case, and the
aforementioned Tauberian conditions are satisfied.

Now, the final part of the proof of \eqref{SP5.4} reduces to an
elementary algebraic calculations:
$$
\gamma_2:=\frac{\mathrm{d}^{2}}{\mathrm{d}z^{2}}\widehat{B}_1(\lambda-\lambda\widehat{R}(z))|_{z=1}
=\frac{\mathsf{E}\varsigma^{2}}{\mathsf{E}\varsigma}-1+\rho_{1,2}(\mathsf{E}\varsigma)^{2}.
$$
The lemma is proved.
\end{proof}

With the aid of Lemma \ref{Asymp behav 1} one can easily obtain the
statements on asymptotic behavior of $\mathsf{E}\nu_L^{(1)}$,
$\mathsf{EE}\{\nu_L^{(1)}|\varsigma \wedge L\}$ and, consequently,
$p_1$ and $p_2$. The theorem below characterizes asymptotic behavior
of the probabilities $p_1$ and $p_2$ as $L\to\infty$.

\begin{thm}\label{M asymp}
If $\rho_1<1$, then
\begin{eqnarray}
\lim_{L\to\infty} p_1(L)&=&1-\rho_1,\label{SP5.6}\\
\lim_{L\to\infty} p_2(L)&=&0.\label{SP5.7}
\end{eqnarray}
If $\rho_1=1$, and additionally $\rho_{1,2}=\int_0^\infty (\lambda
x)^2\mbox{d}B_1(x)<\infty$ and $\mathsf{E}\varsigma^{2}<\infty$,
then
\begin{eqnarray}
\lim_{L\to\infty}
Lp_1(L)&=&\frac{\rho_{1,2}(\mathsf{E}\varsigma)^{3}+\mathsf{E}\varsigma^{2}-\mathsf{E}\varsigma
}{2\mathsf{E}\varsigma},\label{SP5.8}\\
\lim_{L\to\infty}
Lp_2(L)&=&\frac{\rho_2}{1-\rho_2}\frac{\rho_{1,2}(\mathsf{E}\varsigma)^{3}+\mathsf{E}\varsigma^{2}-\mathsf{E}\varsigma
}{2\mathsf{E}\varsigma}.\label{SP5.9}
\end{eqnarray}
If $\rho_1>1$, then
\begin{eqnarray}
\lim_{L\to\infty}\frac{p_1(L)}{\varphi^L}
&=&\frac{(1-\rho_2)[1+\lambda\widehat{B}_1^\prime(\lambda-\lambda\widehat{R}(\varphi))\widehat{R}^\prime(\varphi)]
(1-\varphi)\mathrm{E}\varsigma}
{(\rho_1-\rho_2)[1-\widehat{R}(\varphi)]},\label{SP5.10}\\
\lim_{L\to\infty}p_2(L)&=&\frac{\rho_2(\rho_1-1)}{\rho_1-\rho_2},\label{SP5.11}
\end{eqnarray}
where $\varphi$ is defined in the formulation of Lemma \ref{Asymp behav 1}.
\end{thm}

\begin{proof} Let us first find asymptotic representation for $\mathsf{EE}\{\nu_L^{(1)}|\varsigma_1\wedge L\}$ as $L\to\infty$. According to Lemma \ref{Asymp behav 1} and explicit representation \eqref{SP4.8} we obtain as follows.

If $\rho_1<1$, then
\begin{equation}\label{SP5.12}
\begin{aligned}
\lim_{L\to\infty}\mathsf{EE}\{\nu_L^{(1)}|\varsigma_1\wedge L\}&=
\frac{1}{1-\rho_1}\lim_{L\to\infty}\sum_{i=1}^{L}i\mathsf{Pr}\{\varsigma_1\wedge L=i\}\\
&=\frac{\mathsf{E}\varsigma}{1-\rho_1}.
\end{aligned}
\end{equation}
If $\rho_1=1$, $\rho_{1,2}<\infty$ and
$\mathsf{E}\varsigma^{2}<\infty$, then
\begin{equation}\label{SP5.13}
\begin{aligned}
\lim_{L\to\infty}\frac{\mathsf{EE}\{\nu_L^{(1)}|\varsigma_1\wedge
L\}}{L}&=\frac{2\mathsf{E}\varsigma}{\rho_{1,2}(\mathsf{E}\varsigma)^{3}+\mathsf{E}\varsigma^{2}
-\mathsf{E}\varsigma}\lim_{L\to\infty}
\sum_{i=1}^{L}i\mathsf{Pr}\{\varsigma_1\wedge L=i\}\\
&=\frac{2(\mathsf{E}\varsigma)^{2}}{\rho_{1,2}(\mathsf{E}\varsigma)^{3}+\mathsf{E}\varsigma^{2}-
\mathsf{E}\varsigma}.
\end{aligned}
\end{equation}

If $\rho_1>1$, then
\begin{equation}\label{SP5.14}
\begin{aligned}
\lim_{L\to\infty}\frac{\mathsf{EE}\{\nu_L^{(1)}|\varsigma_1\wedge
L\}}{\varphi^L}&=
\frac{1}{1+\lambda\widehat{B}_1^\prime(\lambda-\lambda\widehat{R}(\varphi))\widehat{R}^\prime(\varphi)}\\
&\ \ \ \times
\lim_{L\to\infty}\sum_{i=1}^L\mathsf{Pr}\{\varsigma_1\wedge L=i\}\sum_{j=0}^{i-1}\varphi^j\\
&=\frac{1}{[1+\lambda\widehat{B}_1^\prime(\lambda-\lambda\widehat{R}(\varphi))\widehat{R}^\prime(\varphi)](1-\varphi)}\\
&\ \ \ \times
\lim_{L\to\infty}\sum_{i=1}^L\mathsf{Pr}\{\varsigma_1\wedge L=i\}(1-\varphi^i)\\
&=\frac{1-\widehat{R}(\varphi)}{[1+\lambda\widehat{B}_1^\prime(\lambda-\lambda\widehat{R}(\varphi))
\widehat{R}^\prime(\varphi)](1-\varphi)}.
\end{aligned}
\end{equation}
Therefore, taking into account these limiting relations
\eqref{SP5.12}, \eqref{SP5.13} and \eqref{SP5.14} by virtue of
\eqref{SP4.17} (Corollary \ref{Asymp estimate}) and explicit
representations \eqref{SP4.19} and \eqref{SP4.20} (Lemma \ref{Prel
asymp}) for $p_1$ and $p_2$, we finally arrive at the statements of
the theorem. The theorem is proved.
\end{proof}

\subsection{Asymptotic theorems for $p_1$ and $p_2$ under special heavy load conditions.}\label{Final asymp 2}

In this section we establish asymptotic theorems for $p_1$ and $p_2$
under heavy load assumptions where (i) $\rho_1=1+\delta$ or (ii)
$\rho_1=1-\delta$, and $\delta$ is a vanishing positive parameter as
$L\to\infty$. The theorems presented in this section are analogues
of the theorems of \cite{Abramov 2007} given in Section 4 of that
paper. The conditions are special, because these heavy load
conditions include the change of the parameter $\rho_1$ as $L$
increases to infinity and $\delta$ vanishes, but the other load
parameter $\rho_2$ remains unchanged.

\smallskip
In case (i) we have the following two theorems.

\begin{thm}\label{i1} Assume that $\rho_1=1+\delta$, $\delta>0$ and
that $L\delta\to C>0$ as $\delta\to0$ and $L\to\infty$. Assume that
$\rho_{1,3}(L)$ is a bounded sequence, assume that
$\mathsf{E}\varsigma^{3}<\infty$ and that the limit
$\lim\limits_{L\to\infty}\rho_{1,2}(L)=\widetilde\rho_{1,2}$ exists.
Then,
\begin{eqnarray}
p_1&=&\frac{\delta}{\exp\left(\dfrac{2C\mathsf{E}\varsigma}
{\widetilde{\rho}_{1,2}(\mathsf{E}\varsigma)^{3}+\mathsf{E}\varsigma^{2}-\mathsf{E}\varsigma}\right)-1}[1+o(1)],\label{SP6.1}\\
p_2&=&\frac{\delta\rho_2\exp\left(\dfrac{2C\mathsf{E}\varsigma}
{\widetilde{\rho}_{1,2}(\mathsf{E}\varsigma)^{3}+\mathsf{E}\varsigma^{2}-\mathsf{E}\varsigma}\right)}
{(1-\rho_2)\left[\exp\left(\dfrac{2C\mathsf{E}\varsigma}
{\widetilde{\rho}_{1,2}(\mathsf{E}\varsigma)^{3}+\mathsf{E}\varsigma^{2}-\mathsf{E}\varsigma}\right)-1\right]}
[1+o(1)].\label{SP6.2}
\end{eqnarray}
\end{thm}
\begin{proof} Note first, that under assumptions of the theorem
there is the following expansion for $\varphi$:
\begin{equation}\label{SP6.3}
\varphi=1-\frac{2\delta\mathsf{E}\varsigma}{\widetilde{\rho}_{1,2}(\mathsf{E}\varsigma)^{3}+\mathsf{E}\varsigma^{2}-\mathsf{E}\varsigma}+O(\delta^{2}).
\end{equation}
This expansion is similar to that given originally in the book of
Subhankulov \cite{Subhankulov 1976}, p.362, and its proof is
provided as follows. Write the equation
$\varphi=\widehat{B}_1(\lambda-\lambda\widehat{R}(\varphi))$ and
expand the right-hand side by Taylor's formula taking $\varphi=1-z$,
where $z$ is small enough, when $\delta$ is small. We obtain:
\begin{equation}\label{SP6.4}
\begin{aligned}
1-z&=1-(1+\delta)z
+\frac{\widetilde\rho_{1,2}(\mathsf{E}\varsigma)^{3}+(1+\delta)\left({\mathsf{E}\varsigma^{2}}-
{\mathsf{E}\varsigma}\right)}{2\mathsf{E}\varsigma}z^{2} +O(z^{3}).
\end{aligned}
\end{equation}
Disregarding the small term $O(z^{3})$ in \eqref{SP6.4} we arrive at
the quadratic equation
\begin{equation}\label{SP6.5}
\delta
z-\frac{\widetilde\rho_{1,2}(\mathsf{E}\varsigma)^{3}+(1+\delta)\left({\mathsf{E}\varsigma^{2}}-
{\mathsf{E}\varsigma}\right)}{2\mathsf{E}\varsigma}z^{2}=0.
\end{equation}
The positive solution of \eqref{SP6.5},
$$z=\frac{2\delta\mathsf{E}\varsigma}{\widetilde\rho_{1,2}(\mathsf{E}\varsigma)^{3}+(1+\delta)\left({\mathsf{E}\varsigma^{2}}-
{\mathsf{E}\varsigma}\right)},$$ leads to the expansion given by
\eqref{SP6.3}.

Let us now expand the right-hand side of \eqref{SP5.14} when
$\delta$ is small. For the term
$1+\lambda\widehat{B}_1^\prime(\lambda-\lambda
\widehat{R}(\varphi))\widehat{R}^\prime(\varphi)$ we have the
expansion
\begin{equation}\label{SP6.6}
1+\lambda\widehat{B}_1^\prime(\lambda-\lambda
\widehat{R}(\varphi))\widehat{R}^\prime(\varphi)=\delta+O(\delta^{2}),
\end{equation}
Then, according to the l'Hospitale rule
$$
\lim_{u\uparrow1}\frac{1-\widehat{R}(u)}{1-u}=\mathsf{E}\varsigma.
$$
Hence
\begin{equation}\label{SP6.7}
\frac{1-\widehat{R}(\varphi)}{1-\varphi}=\mathsf{E}\varsigma[1+o(1)].
\end{equation}

 Substituting \eqref{SP6.3}, \eqref{SP6.6} and \eqref{SP6.7} into
 \eqref{SP5.14} we obtain the expansion
\begin{equation}\label{SP6.8}
\mathsf{EE}\{\nu_L^{(1)}|\varsigma\wedge
L\}=\frac{\exp\left(\dfrac{2C\mathsf{E}\varsigma}{\widetilde{\rho}_{1,2}(\mathsf{E}\varsigma)^{3}
+\mathsf{E}\varsigma^{2}-\mathsf{E}\varsigma}\right)-1}{\delta}\mathsf{E}\varsigma[1+o(1)].
\end{equation}
Hence, relations \eqref{SP6.1} and \eqref{SP6.2} of the theorem
follow by virtue of \eqref{SP4.17} (Corollary \ref{Asymp estimate})
and explicit representations \eqref{SP4.19} and \eqref{SP4.20}
(Lemma \ref{Prel asymp}) for $p_1$ and $p_2$.
\end{proof}

\begin{thm}\label{i2} Under the conditions of Theorem \ref{i1}
assume that $C=0$. Then,
\begin{eqnarray}
\lim_{L\to\infty}Lp_1(L)&=&\frac
{\widetilde{\rho}_{1,2}(\mathsf{E}\varsigma)^{3}+\mathsf{E}\varsigma^{2}-\mathsf{E}\varsigma}{2\mathsf{E}\varsigma},\label{SP6.9}\\
\lim_{L\to\infty}Lp_2(L)&=&\frac{\rho_2}{1-\rho_2}\frac
{\widetilde{\rho}_{1,2}(\mathsf{E}\varsigma)^{3}+\mathsf{E}\varsigma^{2}-\mathsf{E}\varsigma}{2\mathsf{E}\varsigma}.\label{SP6.10}
\end{eqnarray}
\end{thm}

\begin{proof} The statement of the theorem follows by expanding the
main terms of asymptotic relations \eqref{SP6.1} and \eqref{SP6.2}
for small $C$.
\end{proof}

\smallskip
In case (ii) we have the following two theorems.

\begin{thm}\label{i3} Assume that $\rho_1=1-\delta$, $\delta>0$ and
that $L\delta\to C>0$ as $\delta\to0$ and $L\to\infty$. Assume that
$\rho_{1,3}(L)$ is a bounded function, assume that
$\mathsf{E}\varsigma^{3}<\infty$ and that the limit
$\lim\limits_{L\to\infty}\rho_{1,2}(L)=\widetilde\rho_{1,2}$ exists.
Then,
\begin{eqnarray}
p_1&=&\delta\exp\left(\frac
{\widetilde{\rho}_{1,2}(\mathsf{E}\varsigma)^{3}+\mathsf{E}\varsigma^{2}-\mathsf{E}\varsigma}{2C\mathsf{E}\varsigma}\right)[1+o(1)],
\label{SP6.11}\\
p_2&=&\frac{\delta\rho_2\left[\exp\left(\dfrac
{\widetilde{\rho}_{1,2}(\mathsf{E}\varsigma)^{3}+\mathsf{E}\varsigma^{2}-\mathsf{E}\varsigma}{2C\mathsf{E}\varsigma}\right)-1\right]}
{1-\rho_2}[1+o(1)].\label{SP6.12}
\end{eqnarray}
\end{thm}

\begin{proof} The explicit representation for the generating
function for $\mathsf{E}\widetilde\nu_j^{(1)}$ is given by
\eqref{SP5.2}. Since the sequence
$\{\mathsf{E}\widetilde\nu_j^{(1)}\}$ is increasing, then the
asymptotic behavior of $\mathsf{E}\nu_L^{(1)}$ as $L\to\infty$ under
the assumptions $\rho_1=1-\delta$, $L\delta\to C$ as $L\to\infty$
can be found according to a Tauberian theorem of Hardy and
Littlewood (see e.g. \cite{Postnikov 1980}, \cite{Subhankulov 1976},
\cite{Sznajder Filar 1992}, \cite{Widder}, and \cite{Takacs 1967},
p.203). Namely, according to that theorem, the behaviour of
$\mathsf{E}\widetilde\nu_L^{(1)}$ as $L\to \infty$ and $\delta\to0$
such that $\delta L\to C>0$ can be found from the asymptotic
expansion of
\begin{equation}\label{SP6.13}
(1-z)\frac{\widehat{B}_1(\lambda-\lambda\widehat{R}(z))}{\widehat{B}_1(\lambda-\lambda\widehat{R}(z))-z}
\end{equation}
as $z\uparrow1$. Similarly to the evaluation given in the proof of Theorem 4.3 \cite{Abramov 2007}, we have:
\begin{equation}\label{SP6.14}
\begin{aligned}
&(1-z)\frac{\widehat{B}_1(\lambda-\lambda\widehat{R}(z))}{\widehat{B}_1(\lambda-\lambda\widehat{R}(z))-z}\\&=
\frac{1-z}{1-z-\rho_1(1-z)+\dfrac
{\widetilde{\rho}_{1,2}(\mathsf{E}\varsigma)^{3}+\mathsf{E}\varsigma^{2}-\mathsf{E}\varsigma}{2\mathsf{E}\varsigma}(1-z)^2+O((1-z)^{3})}\\
&=\frac{1}{\delta+\dfrac
{\widetilde{\rho}_{1,2}(\mathsf{E}\varsigma)^{3}+\mathsf{E}\varsigma^{2}-\mathsf{E}\varsigma}{2\mathsf{E}\varsigma}(1-z)+O((1-z)^{2})}\\
&=\frac{1}{\delta\left[1+\dfrac
{\widetilde{\rho}_{1,2}(\mathsf{E}\varsigma)^{3}+\mathsf{E}\varsigma^{2}-\mathsf{E}\varsigma}{2\delta\mathsf{E}\varsigma}(1-z)\right]+O((1-z)^{2})}\\
&=\frac{1}{\delta\exp\left[\dfrac
{\widetilde{\rho}_{1,2}(\mathsf{E}\varsigma)^{3}+\mathsf{E}\varsigma^{2}-\mathsf{E}\varsigma}{2\delta\mathsf{E}\varsigma}(1-z)\right]}[1+o(1)].
\end{aligned}
\end{equation}
Therefore, assuming that $z=\frac{L-1}{L}\to1$ as $L\to\infty$, from \eqref{SP6.14} we arrive at the following estimate:
\begin{equation}\label{SP6.15}
\mathsf{E}\widetilde\nu_L^{(1)}=\frac{1}{\delta}\exp\left(-\frac
{\widetilde{\rho}_{1,2}(\mathsf{E}\varsigma)^{3}+\mathsf{E}\varsigma^{2}-\mathsf{E}\varsigma}{2C\mathsf{E}\varsigma}\right)[1+o(1)].
\end{equation}

Comparing \eqref{SP5.5} with \eqref{SP5.14} and taking into account
\eqref{SP6.7}, which holds true in the case of this theorem as well,
we obtain:
 \begin{equation}\label{SP6.16}
\mathsf{EE}\{\nu_L^{(1)}|\varsigma_1\wedge
L\}=\frac{\mathsf{E}\varsigma}{\delta}\exp\left(-\frac
{\widetilde{\rho}_{1,2}(\mathsf{E}\varsigma)^{3}+\mathsf{E}\varsigma^{2}-\mathsf{E}\varsigma}{2C\mathsf{E}\varsigma}\right)[1+o(1)].
\end{equation}
Hence, relations \eqref{SP6.11} and \eqref{SP6.12} of the theorem
follow by virtue of \eqref{SP4.17} (Corollary \ref{Asymp estimate})
and explicit representations \eqref{SP4.19} and \eqref{SP4.20}
(Lemma \ref{Prel asymp}) for $p_1$ and $p_2$.
\end{proof}

\begin{thm}\label{i4}
Under the conditions of Theorem \ref{i3}
assume that $C=0$. Then we have \eqref{SP6.9} and \eqref{SP6.10}.
\end{thm}

\begin{proof} The proof of the theorem follows by expanding the main terms of the asymptotic relations \eqref{SP6.11} and \eqref{SP6.12} for small $C$.
\end{proof}

\section{Asymptotic theorems for the stationary probabilities $q_i$}\label{Q stationary
probabilities}

The aim of this section is asymptotic analysis of the stationary
probabilities $q_i$, $i=1,2,\ldots,L$ as $L\to\infty$. The challenge
is to first obtain the explicit representation for $q_i$ in terms of
$\mathsf{E}\nu_i^{(1)}$, and then to study the asymptotic behavior
of $q_i$ as $L\to\infty$ on the basis of the known asymptotic
results for $\mathsf{E}\nu_i^{(1)}$ as $L\to\infty$. The asymptotic
results are obtained in the following three cases: $\rho_1=1$,
$\rho_1=1+\delta$ and $\rho_1=1-\delta$, where $\delta$ is a
positive small value.

\subsection{Explicit representation for the stationary probabilities $q_i$.}\label{Explicit q}
The aim of this section is to prove the following statement.

\begin{lem}\label{Explicit qi} For $i=1,2,\ldots,L$ we have
\begin{equation}\label{SP7.1}
q_i=\rho_1p_1\left(\mathsf{E}\nu_i^{(1)}-\mathsf{E}\nu_{i-1}^{(1)}\right).
\end{equation}
\end{lem}

\begin{proof}
Using renewal arguments (e.g. \cite{Ross 2000}), relation \eqref{MXG1.0} and Wald's identities:
\begin{equation*}
\mathsf{E}T_i^{(1)}=\frac{\rho_1}{\lambda\mathsf{E}\varsigma}\mathsf{E}\nu_i^{(1)},
\ i=1,2,\ldots,L,
\end{equation*}
 we have:
\begin{equation}\label{SP7.2}
q_i=\frac{\mathsf{E}T_i^{(1)}-\mathsf{E}T_{i-1}^{(1)}}{\mathsf{E}T_L+\dfrac{1}{\lambda}}=\rho_1
\frac{\mathsf{E}\nu_i^{(1)}-\mathsf{E}\nu_{i-1}^{(1)}}{\mathsf{E}\nu_L},
\ i=1,2,\ldots,L.
\end{equation}
Taking into account that
$\mathsf{E}\nu_L=\mathsf{E}\nu_L^{(1)}+\mathsf{E}\nu_L^{(2)}$ and
then applying the linear representation for $\mathsf{E}\nu_L^{(2)}$
given by \eqref{SP4.23}, from \eqref{SP7.2} we obtain:
\begin{equation*}\label{SP7.3}
q_i=\frac{\rho_1(1-\rho_2)}{\mathsf{E}\varsigma+(\rho_1-\rho_2)\mathsf{E}\nu_L^{(1)}}
\left(\mathsf{E}\nu_i^{(1)}-\mathsf{E}\nu_{i-1}^{(1)}\right), \
i=1,2,\ldots,L.
\end{equation*}
Hence, representation \eqref{SP7.1} follows from \eqref{SP4.19} (Lemma \ref{Prel asymp}), and Lemma \ref{Explicit qi} is proved.
\end{proof}

\subsection{Asymptotic analysis of the stationary probabilities $q_i$: The case
$\rho_1=1$.}\label{Case 1} Let us study asymptotic behavior of the
stationary probabilities $q_i$. We start from the following modified
version of \eqref{SP5.4} (Lemma \ref{Asymp behav 1}):
\begin{equation}\label{SP8.1}
 \mathsf{E}\widetilde\nu_{L-j}^{(1)}-\mathsf{E}\widetilde\nu_{L-j-1}^{(1)}
=\frac
{2\mathsf{E}\varsigma}{\rho_{1,2}(\mathsf{E}\varsigma)^{3}+\mathsf{E}\varsigma^{2}-\mathsf{E}\varsigma}+o(1),
\end{equation}
which is assumed to be satisfied under the conditions $\rho_{1,2}<\infty$ and $\mathrm{E}\varsigma^{2}<\infty$. Under the same conditions, similarly to \eqref{SP5.13} we obtain:
\begin{equation}\label{SP8.2}
\begin{aligned}
\mathsf{EE}\{\nu_{L-j}^{(1)}|\varsigma_1\wedge
L\}-\mathsf{EE}\{\nu_{L-j-1}^{(1)}|\varsigma_1\wedge L\}&= \frac
{2\mathsf{E}\varsigma}{\rho_{1,2}(\mathsf{E}\varsigma)^{3}+\mathsf{E}\varsigma^{2}-\mathsf{E}\varsigma}\\
&\ \ \ \times
\sum_{i=1}^{L}i\mathsf{Pr}\{\varsigma_1\wedge L=i\}+o(1)\\
&=\frac
{2(\mathsf{E}\varsigma)^{2}}{\rho_{1,2}(\mathsf{E}\varsigma)^{3}+\mathsf{E}\varsigma^{2}-\mathsf{E}\varsigma}+o(1).
\end{aligned}
\end{equation}
Hence, according to \eqref{SP4.17} (Corollary \ref{Asymp estimate}) and \eqref{SP8.2} we have the estimate
\begin{equation}\label{SP8.3}
\begin{aligned}
\mathsf{E}\nu_{L-j}^{(1)}-\mathsf{E}\nu_{L-j-1}^{(1)} &=\frac
{2(\mathsf{E}\varsigma)^{2}}{\rho_{1,2}(\mathsf{E}\varsigma)^{3}+\mathsf{E}\varsigma^{2}-\mathsf{E}\varsigma}+o(1).
\end{aligned}
\end{equation}

Asymptotic relations \eqref{SP8.3}, \eqref{SP5.8} together with
explicit relation \eqref{SP7.1} of Lemma \ref{Explicit qi} leads to
the following theorem.

\begin{thm}\label{lem3} In the case $\rho_1=1$ under the additional conditions $\rho_{1,2}<\infty$
and $\mathrm{E}\varsigma^{2}<\infty$ for any $j\geq0$ we have
\begin{equation}\label{SP8.4}
\lim_{L\to\infty}Lq_{L-j}=1.
\end{equation}
\end{thm}

Note, that the asymptotic relation given by \eqref{SP8.4} is not
expressed via $\mathsf{E}\varsigma$ and, therefore, it is invariant
and hence the same as for the queueing system with ordinary Poisson
arrivals.

\subsection{Asymptotic analysis of the stationary probabilities $q_i$: The case $\rho_1=1+\delta$, $\delta>0$.}
\label{Case 2} In the case $\rho_1=1+\delta$, $\delta>0$ the
asymptotic behaviour of $q_i$ is specified by the following theorem.

\begin{thm}\label{thm1}
Assume that $\rho_1=1+\delta$, $\delta>0$, and $L\delta\to C>0$ as
$\delta\to 0$ and $L\to\infty$. Assume that $\rho_{1,3}(L)$ is a
bounded sequence, assume that $\mathrm{E}\varsigma^{3}<\infty$ and
there exists
$\widetilde{\rho}_{1,2}=\lim\limits_{L\to\infty}\rho_{1,2}(L)$.
Then, for any $j\ge0$
\begin{equation}\label{SP9.1}
\begin{aligned}
q_{L-j}&=\frac{\exp\left(\dfrac{2C\mathsf{E}\varsigma}{\widetilde{\rho}_{1,2}(\mathsf{E}\varsigma)^{3}+\mathsf{E}\varsigma^{2}-\mathsf{E}\varsigma}\right)}
{\exp\left(\dfrac{2C\mathsf{E}\varsigma}{\widetilde{\rho}_{1,2}(\mathsf{E}\varsigma)^{3}+\mathsf{E}\varsigma^{2}-\mathsf{E}\varsigma}\right)-1}\\
&\ \ \ \times
\left(1-\frac{2\delta\mathsf{E}\varsigma}{\widetilde{\rho}_{1,2}(\mathsf{E}\varsigma)^{3}+\mathsf{E}\varsigma^{2}-\mathsf{E}\varsigma}\right)^j
\frac{2\delta\mathsf{E}\varsigma}{\widetilde{\rho}_{1,2}(\mathsf{E}\varsigma)^{3}+\mathsf{E}\varsigma^{2}-\mathsf{E}\varsigma}
+o(\delta).
\end{aligned}
\end{equation}
\end{thm}

\begin{proof} Expanding \eqref{SP5.5} for large $L$, we have:
\begin{equation}\label{SP9.2}
\mathsf{E}\widetilde\nu_{L-j}^{(1)}=\frac{\varphi^j}{\varphi^L[1+\lambda
\widehat{B}_1^\prime(\lambda-\lambda\widehat{R}(\varphi))\widehat{R}^\prime(\varphi)]}+\frac{1}{1-\rho_1}+o(1).
\end{equation}
In turn, from \eqref{SP9.2} for large $L$ we obtain:
\begin{equation}\label{SP9.3}
\mathsf{E}\widetilde\nu_{L-j}^{(1)}-\mathsf{E}\widetilde\nu_{L-j-1}^{(1)}=\frac{(1-\varphi)\varphi^j}
{\varphi^L[1+\lambda
\widehat{B}_1^\prime(\lambda-\lambda\widehat{R}(\varphi))\widehat{R}^\prime(\varphi)]}+o(1).
\end{equation}
From \eqref{SP9.3}, similarly to \eqref{SP5.14}, we further have:
\begin{equation*}
\begin{aligned}
&\mathsf{EE}\{\nu_{L-j}^{(1)}|\varsigma_1\wedge L\}-\mathsf{EE}\{\nu_{L-j-1}^{(1)}|\varsigma_1\wedge L\}\\
&=\frac{(1-\widehat{R}(\varphi))(1-\varphi)\varphi^j}{[1+\lambda\widehat{B}_1^\prime(\lambda-\lambda\widehat{R}(\varphi))
\widehat{R}^\prime(\varphi)](1-\varphi)}+o(1),
\end{aligned}
\end{equation*}
and, according to \eqref{SP4.17} (Corollary \ref{Asymp estimate}),
\begin{equation}\label{SP9.5}
\begin{aligned}
&\mathsf{E}\nu_{L-j}^{(1)}-\mathsf{E}\nu_{L-j-1}^{(1)}
=\frac{(1-\widehat{R}(\varphi))(1-\varphi)\varphi^j}{[1+\lambda\widehat{B}_1^\prime(\lambda-\lambda\widehat{R}(\varphi))
\widehat{R}^\prime(\varphi)](1-\varphi)}+o(1).
\end{aligned}
\end{equation}

Next, under the conditions of the theorem, asymptotic expansions \eqref{SP6.3} \eqref{SP6.6} and \eqref{SP6.7} hold.
Taking into consideration these expansions we arrive at the following asymptotic relations
for $j=0,1,\ldots$:
 \begin{equation*}\label{SP9.7}
 \begin{aligned}
 \mathsf{E}\nu_{L-j}^{(1)}-\mathsf{E}\nu_{L-j-1}^{(1)}&=
 \exp\left(\frac{2C\mathsf{E}\varsigma}{\widetilde{\rho}_{1,2}(\mathsf{E}\varsigma)^{3}+\mathsf{E}\varsigma^{2}-\mathsf{E}\varsigma}\right)\\
& \ \ \ \times
 \left(1-\frac{2\delta\mathsf{E}\varsigma}{\widetilde{\rho}_{1,2}(\mathsf{E}\varsigma)^{3}+\mathsf{E}\varsigma^{2}-\mathsf{E}\varsigma}\right)^j
 \frac{2\delta\mathsf{E}\varsigma}{\widetilde{\rho}_{1,2}(\mathsf{E}\varsigma)^{3}+\mathsf{E}\varsigma^{2}-\mathsf{E}\varsigma}[1+o(1)].
 \end{aligned}
 \end{equation*}
 Now, taking into account asymptotic relation \eqref{SP6.1} of Theorem \ref{i1} and the explicit formula given by \eqref{SP7.1} (Lemma \ref{Explicit qi}) we arrive at the statement of the theorem.
\end{proof}

\subsection{Asymptotic analysis of the stationary probabilities $q_i$: The case $\rho_1=1-\delta$,
$\delta>0$.}\label{Case 3} In the case $\rho_1=1-\delta$,
$\delta>0$, the study is more delicate and based on special
analysis. The additional assumption of this case is that the class
of probability distribution functions $\{B_1(x)\}$ and
$\mathrm{Pr}\{\varsigma=i\}$ are given such that there exists a
unique root $\tau>1$ of the equation
\begin{equation}\label{SP10.1}
z=\widehat{B}_1(\lambda-\lambda \widehat{R}(z)),
\end{equation}
and there exists the first derivative $\widehat{B}_1^\prime(\lambda-\lambda \widehat{R}(\tau))$.

Under the assumption that $\rho_1<1$ the unique root of
\eqref{SP10.1} is not necessarily exists. Such type of condition has
been considered by Willmot \cite{Willmot 1988} to obtain the
asymptotic behavior for high queue-level probabilities in stationary
$M/GI/1$ queues. Denote the stationary probabilities in the $M/GI/1$
queueing system by $q_i[M/GI/1]$, $i=0,1,\ldots$. It was shown in
\cite{Willmot 1988} that
\begin{equation}\label{SP10.2}
q_i[M/GI/1]=\frac{(1-\rho_1)(1-\tau)}{\tau^i[1+\lambda\widehat{B}_1^\prime(\lambda-\lambda\tau)]}[1+o(1)] \ \mbox{as}
\ i\to\infty,
\end{equation}
where $\widehat{B}_1(s)$ denotes the Laplace-Stieltjes transform of the service time distribution in the $M/G/1$ queueing system, and $\tau$ denotes a root of the equation $z=\widehat{B}_1(\lambda-\lambda z)$ greater than 1, which is assumed to be unique.
On the other hand, according to the Pollaczek-Khintchine formula (e.g. Tak\'acs \cite{Takacs 1962}, p.242), $q_i[M/GI/1]$ can be represented explicitly
\begin{equation}\label{SP10.3}
q_i[M/GI/1]=(1-\rho_1)\left(\mathsf{E}\nu_i^{(1)}-\mathsf{E}\nu_{i-1}^{(1)}\right),
i=1,2,\ldots,
\end{equation}
where the random variable $\nu_i^{(1)}$ in this formula is
associated with the number of served customers during a busy period
of the state dependent $M/G/1$ queueing system, where the value of
the system parameter, where the service is changed, is $i$ (see
Section \ref{MG1}). Representation \eqref{SP10.3} can be easily
checked, since in this case
\begin{equation}\label{SP10.4}
\sum_{j=0}^\infty\mathsf{E}\nu_j^{(1)}z^j=\frac{\widehat{B}_1(\lambda-\lambda
z)}{\widehat{B}_1(\lambda-\lambda z)-z},
\end{equation}
and multiplication of the right-hand side of \eqref{SP10.4} by
$(1-\rho_1)(1-z)$ leads to the well-known Pollaczek-Khintchine
formula. Then, from \eqref{SP10.2} and \eqref{SP10.3} there is the
asymptotic proportion for large $L$ and any $j\geq0$:
\begin{equation}\label{SP10.5}
\frac{\mathsf{E}\nu_{L-j}^{(1)}-\mathsf{E}\nu_{L-j-1}^{(1)}}{\mathsf{E}\nu_{L}^{(1)}-\mathsf{E}\nu_{L-1}^{(1)}}
=\tau^j[1+o(1)].
\end{equation}

In the case of batch arrivals the results are similar. One can prove
that the same proportion as \eqref{SP10.5} holds in this case as
well, where $\tau$ in the case of batch arrivals denotes a unique
real root of the equation of \eqref{SP10.1}, which is greater than
1. (Recall that our convention is an existence of a unique real
solution of \eqref{SP10.1} greater than 1.) Indeed, the arguments of
\cite{Willmot 1988} are elementary extended for the queueing system
with batch arrivals. The simplest way to extend these results
straightforwardly is to consider the stationary queueing system with
batch Poisson arrivals, in which the first batch in each busy period
is equal to 1. Denote this system by $M^{1,X}/G/1$. For this
specific system, similarly to \eqref{SP10.2} we obtain:
\begin{equation}\label{SP10.6}
q_i[M^{1,X}/GI/1]=
\frac{(1-\rho_1)(1-\tau)}{\tau^i[1+\lambda\widehat{B}_1^\prime(\lambda-\lambda
\widehat{R}(\tau))\widehat{R}^\prime(\tau)]}[1+o(1)] \ \mbox{as} \
i\to\infty,
\end{equation}
where $q_i[M^{1,X}/GI/1]$, $i=0,1,\ldots$,  denotes the stationary
probabilities in this system. Then, taking into account
\eqref{SP5.2}, similarly to \eqref{SP10.3} one can write
\begin{equation}\label{SP10.7}
q_i[M^{1,X}/GI/1]=(1-\rho_1)\left(\mathsf{E}\widetilde\nu_i^{(1)}-\mathsf{E}\widetilde\nu_{i-1}^{(1)}\right),
\ i=1,2,\ldots.
\end{equation}
From \eqref{SP10.6} and \eqref{SP10.7} we obtain
\begin{equation}\label{SP10.8}
\frac{\mathsf{E}\widetilde\nu_{L-j}^{(1)}-\mathsf{E}\widetilde\nu_{L-j-1}^{(1)}}
{\mathsf{E}\widetilde\nu_{L}^{(1)}-\mathsf{E}\widetilde\nu_{L-1}^{(1)}}
=\tau^j[1+o(1)].
\end{equation}
From \eqref{SP10.8} and the results of Sections \ref{Preliminary asymp} and \ref{Final asymp} (see
\eqref{SP5.3}, \eqref{SP5.12} and \eqref{SP4.17}) we also have
the estimate
\begin{equation}\label{SP10.9}
\frac{\mathsf{E}\nu_{L-j}^{(1)}-\mathsf{E}\nu_{L-j-1}^{(1)}}
{\mathsf{E}\nu_{L}^{(1)}-\mathsf{E}\nu_{L-1}^{(1)}} =\tau^j[1+o(1)],
\end{equation}
which coincides with \eqref{SP10.5}.

\smallskip
Now we formulate and prove a theorem on asymptotic behavior of the stationary probabilities $q_i$ in the case $\rho_1=1-\delta$, $\delta>0$. The special assumption in this theorem is that the class of probability distributions $\{B_1(x)\}$ is defined according to the above convention. More precisely, in the case $\rho_1=1-\delta$, $\delta>0$, and vanishing $\delta$ as $L\to\infty$ this means that there exists $\epsilon_0>0$ (small enough) such that for all $0\leq\epsilon\leq\epsilon_0$, the above family of probability distribution functions $B_{1,\epsilon}(x)$ (depending now on the parameter $\epsilon$) satisfies the following properties. Let $\widehat{B}_{1,\epsilon}(s)$ denote the Laplace-Stieltjes transform of $B_{1,\epsilon}(x)$. We assume that any $\widehat{B}_{1,\epsilon}(s)$ is an analytic function in a small neighborhood of zero, and
\begin{equation}\label{SP10.10}
\widehat{B}_{1,\epsilon}^\prime(s)<\infty.
\end{equation}
Property \eqref{SP10.10} is required for the existence of the probabilities $q_i$. Relation \eqref{SP10.6} contains the term $\widehat{B}_1^\prime(\lambda-\lambda
\widehat{R}(\tau))$, and this term must be finite. In addition, the term
$
\widehat{R}^\prime(\tau)<\infty
$
must be finite
 as well, that is, the additional to \eqref{SP10.10} associated assumption is that
 \begin{equation}\label{SP10.11}
 \widehat{R}^\prime(1+\epsilon)<\infty
\end{equation}
for any $\epsilon$ of the defined neighborhood. Choice of small parameter $\epsilon$ is continuously connected with that choice of the parameter $\delta$ (or $L$) in the theorem below.

\begin{thm}\label{thm4} Assume that the class of probability distribution functions
$\{B_1(x)\}$ and the probabilities $r_1$, $r_2$, \ldots are defined
according to the conventions made and, respectively, satisfy
\eqref{SP10.10} and \eqref{SP10.11}, $\rho_1=1-\delta$, $\delta>0$,
and $L\delta\to C>0$, as $\delta\to0$ and $L\to\infty$. Assume that
$\rho_{1,3}=\rho_{1,3}(L)$ is a bounded sequence,
$\mathrm{E}\varsigma^{3}<\infty$, and there exists
$\widetilde\rho_{1,2}=\lim\limits_{L\to\infty}\rho_{1,2}(L)$. Then,
\begin{equation}\label{SP10.12}
\begin{aligned}
q_{L-j}&=\frac{1}{\exp\left(\dfrac{2C\mathsf{E}\varsigma}{\widetilde{\rho}_{1,2}(\mathsf{E}\varsigma)^{3}+\mathsf{E}\varsigma^{2}-\mathsf{E}\varsigma}\right)-1}\\
&\ \ \ \times
\frac{2\delta\mathsf{E}\varsigma}{\widetilde{\rho}_{1,2}(\mathsf{E}\varsigma)^{3}+\mathsf{E}\varsigma^{2}-\mathsf{E}\varsigma}
\left(1+\frac{2\delta\mathsf{E}\varsigma}{\widetilde{\rho}_{1,2}(\mathsf{E}\varsigma)^{3}+\mathsf{E}\varsigma^{2}-\mathsf{E}\varsigma}\right)^j[1+o(1)]
\end{aligned}
\end{equation}
for any $j\geq0.$
\end{thm}

\begin{proof}
Under the assumptions of this theorem let us first derive the following asymptotic expansion:
\begin{equation}\label{SP10.13}
\tau=1+\frac{2\delta\mathsf{E}\varsigma}{\widetilde{\rho}_{1,2}(\mathsf{E}\varsigma)^{3}+\mathsf{E}\varsigma^{2}-\mathsf{E}\varsigma}+O(\delta^{2}).
\end{equation}
Asymptotic expansion \eqref{SP10.13} is similar to that of \eqref{SP6.3}, and its proof is also similar. Namely, taking into account that the equation $z=\widehat{B}_1(\lambda-\lambda\widehat{R}(z))$  has a unique solution in the set $(1,\infty)$, and this solution approaches 1 as $\delta$ vanishes. Therefore, by the Taylor expansion of this equation around the point $z=1$, we have:
\begin{equation}\label{SP10.14}
1+z=1+(1-\delta)z+\frac{\widetilde{\rho}_{1,2}(\mathsf{E}\varsigma)^{3}+(1-\delta)(\mathsf{E}\varsigma^{2}-\mathsf{E}\varsigma)}
{2\mathsf{E}\varsigma}z^{2}+O(z^{3}).
\end{equation}
Disregarding the term $O(z^{3})$, from \eqref{SP10.14} we arrive at
the quadratic equation
$$
\delta
z-\frac{\widetilde{\rho}_{1,2}(\mathsf{E}\varsigma)^{3}+(1-\delta)(\mathsf{E}\varsigma^{2}-\mathsf{E}\varsigma)}
{2\mathsf{E}\varsigma}z^{2}=0,
$$
and obtain a positive solution
$$z=\frac{2\delta\mathsf{E}\varsigma}{\widetilde\rho_{1,2}(\mathsf{E}\varsigma)^{3}+(1+\delta)\left({\mathsf{E}\varsigma^{2}}-
{\mathsf{E}\varsigma}\right)}.$$
This proves \eqref{SP10.13}.

Next, from \eqref{SP10.9}, \eqref{SP10.13} and explicit formula \eqref{SP7.1} we obtain
\begin{equation}\label{SP10.15}
q_{L-j}=q_L\left(1+\frac{2\delta\mathsf{E}\varsigma}{\widetilde{\rho}_{1,2}(\mathsf{E}\varsigma)^{3}+\mathsf{E}\varsigma^{2}-\mathsf{E}\varsigma}\right)^j[1+o(1)].
\end{equation}
Taking into consideration
\begin{equation*}\label{SP10.16}
\begin{aligned}
&\sum_{j=0}^{L-1}\left(1+\frac{2\delta\mathsf{E}\varsigma}{\widetilde{\rho}_{1,2}(\mathsf{E}\varsigma)^{3}+\mathsf{E}\varsigma^{2}-\mathsf{E}\varsigma}\right)^j\\
&=\frac{\widetilde{\rho}_{1,2}(\mathsf{E}\varsigma)^{3}+\mathsf{E}\varsigma^{2}-\mathsf{E}\varsigma}{2\delta\mathsf{E}\varsigma}\left[\left(1+
\frac{2\delta\mathsf{E}\varsigma}{\widetilde{\rho}_{1,2}(\mathsf{E}\varsigma)^{3}+\mathsf{E}\varsigma^{2}-\mathsf{E}\varsigma}\right)^L-1\right]\\
&=\frac{\widetilde{\rho}_{1,2}(\mathsf{E}\varsigma)^{3}+\mathsf{E}\varsigma^{2}-\mathsf{E}\varsigma}{2\delta\mathsf{E}\varsigma}\left[\exp
\left(\frac{2C\mathsf{E}\varsigma}{\widetilde{\rho}_{1,2}(\mathsf{E}\varsigma)^{3}+\mathsf{E}\varsigma^{2}-\mathsf{E}\varsigma}\right)-1\right][1+o(1)],
\end{aligned}
\end{equation*}
from the normalization condition $p_1+p_2+\sum\limits_{i=1}^Lq_i=1$ and the fact that both $p_1$ and $p_2$ have the order $O(\delta)$, we obtain:
\begin{equation}\label{SP10.17}
q_L=\frac{2\delta\mathsf{E}\varsigma}{\widetilde{\rho}_{1,2}(\mathsf{E}\varsigma)^{3}+\mathsf{E}\varsigma^{2}-\mathsf{E}\varsigma}
\cdot\frac{1}{ \exp
\left(\dfrac{2C\mathsf{E}\varsigma}{\widetilde{\rho}_{1,2}(\mathsf{E}\varsigma)^{3}+\mathsf{E}\varsigma^{2}-\mathsf{E}\varsigma}\right)-1}[1+o(1)].
\end{equation}
The desired statement of the theorem follows from \eqref{SP10.17}.
\end{proof}

\section{Objective function}\label{ObFunction}

\subsection{The case $\rho_1=1$.}\label{Case 1O} In this section we prove the following result.

\begin{prop}\label{prop1}
In the case $\rho_1=1$, under the additional conditions $\rho_{1,2}<\infty$ and $\mathrm{E}\varsigma^{2}<\infty$ we have:
\begin{equation}\label{SP11.1}
\lim_{L\to\infty}J(L)=j_1\frac{{\rho}_{1,2}(\mathsf{E}\varsigma)^{3}+\mathsf{E}\varsigma^{2}-\mathsf{E}\varsigma}{2\mathsf{E}\varsigma}+j_2\frac{\rho_2}{1-\rho_2}\cdot
\frac{{\rho}_{1,2}(\mathsf{E}\varsigma)^{3}+\mathsf{E}\varsigma^{2}-\mathsf{E}\varsigma}{2\mathsf{E}\varsigma}+c^{*},
\end{equation}
where
\begin{equation*}\label{SP11.2}
c^{*}=\lim_{L\to\infty}\frac{1}{L}\sum_{i=1}^L c_i.
\end{equation*}
\end{prop}

\begin{proof}
The first two terms in the right-hand side of \eqref{SP11.1} follow from asymptotic relations \eqref{SP5.8} and \eqref{SP5.9} (Theorem \ref{M asymp}). The last term $c^*$ of the right-hand side of \eqref{SP11.1} follows from \eqref{SP8.4} (Theorem \ref{lem3}), since
\begin{equation*}\label{SP11.3}
\lim_{L\to\infty}\sum_{i=1}^Lq_ic_i=\lim_{L\to\infty}\frac{1}{L}\sum_{i=1}^L c_i=c^*.
\end{equation*}
\end{proof}

\subsection{The case $\rho=1+\delta$, $\delta>0$.}\label{Case 2O}
 In the case $\rho=1+\delta$, $\delta>0$ we have the following statement.

\begin{prop}\label{prop2} Under the assumptions of Theorem \ref{thm1}
denote the objective function $J$ by $J^{\mathrm{upper}}$. We have the
following representation:
\begin{equation}\label{SP11.4}
\begin{aligned}
J^{\mathrm{upper}}&=C\left[j_1\frac{1}{\exp\left(\dfrac{2C\mathsf{E}\varsigma}{\widetilde{\rho}_{1,2}(\mathsf{E}\varsigma)^{3}+\mathsf{E}\varsigma^{2}-\mathsf{E}\varsigma}\right)-1}\right.\\
&\ \ \
\left.+j_2\frac{\rho_2\exp\left(\dfrac{2C\mathsf{E}\varsigma}{\widetilde{\rho}_{1,2}(\mathsf{E}\varsigma)^{3}+\mathsf{E}\varsigma^{2}-\mathsf{E}\varsigma}\right)}
{(1-\rho_2)\left({\exp\left(\dfrac{2C\mathsf{E}\varsigma}{\widetilde{\rho}_{1,2}(\mathsf{E}\varsigma)^{3}+\mathsf{E}\varsigma^{2}-\mathsf{E}\varsigma}\right)-1}\right)}\right]\\
&\ \ \ +c^{\mathrm{upper}},
\end{aligned}
\end{equation}
where
\begin{equation}\label{SP11.5}
\begin{aligned}
c^{\mathrm{upper}}&=\frac{2C\mathsf{E}\varsigma}{\widetilde{\rho}_{1,2}(\mathsf{E}\varsigma)^{3}+\mathsf{E}\varsigma^{2}-\mathsf{E}\varsigma}\cdot
\frac{\exp\left(\dfrac{2C\mathsf{E}\varsigma}{\widetilde{\rho}_{1,2}(\mathsf{E}\varsigma)^{3}+\mathsf{E}\varsigma^{2}-\mathsf{E}\varsigma}\right)}
{\exp\left(\dfrac{2C\mathsf{E}\varsigma}{\widetilde{\rho}_{1,2}(\mathsf{E}\varsigma)^{3}+\mathsf{E}\varsigma^{2}-\mathsf{E}\varsigma}\right)-1}\\
&\ \ \
\times\lim_{L\to\infty}\frac{1}{L}~\widehat{C}_L\left(1-\frac{2C\mathsf{E}\varsigma}{(\widetilde{\rho}_{1,2}(\mathsf{E}\varsigma)^{3}+\mathsf{E}\varsigma^{2}-\mathsf{E}\varsigma)L}\right),
\end{aligned}
\end{equation}
and $\widehat C_L(z)=\sum\limits_{j=0}^{L-1} c_{L-j}z^j$ is a backward
generating cost function.
\end{prop}

\begin{proof}
The representation for the term
$$
\begin{aligned}
&C\left[j_1\frac{1}{\exp\left(\dfrac{2C\mathsf{E}\varsigma}{\widetilde{\rho}_{1,2}(\mathsf{E}\varsigma)^{3}+\mathsf{E}\varsigma^{2}-\mathsf{E}\varsigma}\right)-1}\right.\\
&\ \ \
\left.+j_2\frac{\rho_2\exp\left(\dfrac{2C\mathsf{E}\varsigma}{\widetilde{\rho}_{1,2}(\mathsf{E}\varsigma)^{3}+\mathsf{E}\varsigma^{2}-\mathsf{E}\varsigma}\right)}
{(1-\rho_2)\left({\exp\left(\dfrac{2C\mathsf{E}\varsigma}{\widetilde{\rho}_{1,2}(\mathsf{E}\varsigma)^{3}+\mathsf{E}\varsigma^{2}-\mathsf{E}\varsigma}\right)-1}\right)}\right]
\end{aligned}
$$
of the right-hand side of \eqref{SP11.4} follows from \eqref{SP6.1}
and \eqref{SP6.2} (Theorem \ref{i1}). This term is similar to that
(5.2) in \cite{Abramov 2007}. The new term which takes into account
the water costs is $c^{\mathrm{upper}}$. Taking into account
representation \eqref{SP9.1}, for this term we obtain:
$$
\begin{aligned}
c^{\mathrm{upper}}&=\lim_{L\to\infty}\sum_{j=0}^{L-1}
q_{L-j}c_{L-j}\\&=\lim_{L\to\infty}\sum_{j=0}^{L-1} c_{L-j}\cdot
\frac{\exp\left(\dfrac{2C\mathsf{E}\varsigma}{\widetilde{\rho}_{1,2}(\mathsf{E}\varsigma)^{3}+\mathsf{E}\varsigma^{2}-\mathsf{E}\varsigma}\right)}
{\exp\left(\dfrac{2C\mathsf{E}\varsigma}{\widetilde{\rho}_{1,2}(\mathsf{E}\varsigma)^{3}+\mathsf{E}\varsigma^{2}-\mathsf{E}\varsigma}\right)-1}\\
&\ \ \ \times\left(1-\frac{2\delta
L\mathsf{E}\varsigma}{(\widetilde{\rho}_{1,2}(\mathsf{E}\varsigma)^{3}+\mathsf{E}\varsigma^{2}-\mathsf{E}\varsigma)L}\right)^j
\frac{2\delta
L\mathsf{E}\varsigma}{(\widetilde{\rho}_{1,2}(\mathsf{E}\varsigma)^{3}+\mathsf{E}\varsigma^{2}-\mathsf{E}\varsigma)L},
\end{aligned}
$$
and, because of $\lim\limits_{L\to\infty}\delta L=C$, representation \eqref{SP11.5} follows.
\end{proof}

\subsection{The case $\rho=1-\delta$, $\delta>0$.}\label{Case 3O}
 In the case $\rho=1-\delta$, $\delta>0$ we have the following statement.

\begin{prop}\label{prop3}
Under the assumptions of Theorem \ref{thm4} denote the objective
function $J$ by $J^{\mathrm{lower}}$. We have the following
representation
\begin{equation}\label{SP12.1}
\begin{aligned}
J^{\mathrm{lower}}&=C\left[j_1\exp\left(\frac{\widetilde{\rho}_{1,2}(\mathsf{E}\varsigma)^{3}+\mathsf{E}\varsigma^{2}-\mathsf{E}\varsigma}
{2C\mathsf{E}\varsigma}\right)\right.\\
&\ \ \ \left.+j_2\frac{\rho_2}{1-\rho_2}
\left(\exp\left(\frac{\widetilde{\rho}_{1,2}(\mathsf{E}\varsigma)^{3}+\mathsf{E}\varsigma^{2}-\mathsf{E}\varsigma}
{2C\mathsf{E}\varsigma}\right)-1\right)\right]\\
&\ \ \ +c^{\mathrm{lower}},
\end{aligned}
\end{equation}
where
\begin{equation}\label{SP12.2}
\begin{aligned}
c^{\mathrm{lower}}&=\frac{2C\mathsf{E}\varsigma}
{\widetilde{\rho}_{1,2}(\mathsf{E}\varsigma)^{3}+\mathsf{E}\varsigma^{2}-\mathsf{E}\varsigma}
\cdot\frac{1}{\exp\left(\dfrac{2C\mathsf{E}\varsigma}
{\widetilde{\rho}_{1,2}(\mathsf{E}\varsigma)^{3}+\mathsf{E}\varsigma^{2}-\mathsf{E}\varsigma}\right)-1}\\
&\ \ \ \times
\lim_{L\to\infty}\frac{1}{L}~\widehat{C}_{L}\left(1+\frac{2C\mathsf{E}\varsigma}
{(\widetilde{\rho}_{1,2}(\mathsf{E}\varsigma)^{3}+\mathsf{E}\varsigma^{2}-\mathsf{E}\varsigma)L}\right),
\end{aligned}
\end{equation}
and $\widehat{C}_L(z)=\sum\limits_{j=0}^{L-1}c_{L-j}z^j$ is a
backward generating cost function.
\end{prop}

\begin{proof}
The representation for the term
\begin{equation*}
C\left[j_1\exp\left(\frac{\widetilde{\rho}_{1,2}(\mathsf{E}\varsigma)^{3}+\mathsf{E}\varsigma^{2}-\mathsf{E}\varsigma}
{2C\mathsf{E}\varsigma}\right)+j_2\frac{\rho_2}{1-\rho_2}
\left(\exp\left(\frac{\widetilde{\rho}_{1,2}(\mathsf{E}\varsigma)^{3}+\mathsf{E}\varsigma^{2}-\mathsf{E}\varsigma}
{2C\mathsf{E}\varsigma}\right)-1\right)\right]
\end{equation*}
of the right-hand side of \eqref{SP12.1} follows from \eqref{SP6.11} and \eqref{SP6.12} (Theorem \ref{i3}).
This term is similar to that (5.3) in \cite{Abramov 2007}. The new term,
which takes into account the water costs, is $c^{\mathrm{lower}}$. Taking into account representation \eqref{SP10.12}, for this
term we obtain:
$$
\begin{aligned}
c^{\mathrm{lower}}&=\lim_{L\to\infty}\sum_{j=0}^{L-1}
q_{L-j}c_{L-j}\\&=\lim_{L\to\infty}\sum_{j=0}^{L-1}
c_{L-j}\cdot\frac{1}{\exp\left(\dfrac{2C\mathsf{E}\varsigma}
{\widetilde{\rho}_{1,2}(\mathsf{E}\varsigma)^{3}+\mathsf{E}\varsigma^{2}-\mathsf{E}\varsigma}\right)-1}\\
&\ \ \ \times \left(1+\frac{2\delta L\mathsf{E}\varsigma}
{(\widetilde{\rho}_{1,2}(\mathsf{E}\varsigma)^{3}+\mathsf{E}\varsigma^{2}-\mathsf{E}\varsigma)L}\right)^j
\frac{2\delta L\mathsf{E}\varsigma}
{(\widetilde{\rho}_{1,2}(\mathsf{E}\varsigma)^{3}+\mathsf{E}\varsigma^{2}-\mathsf{E}\varsigma)L},
\end{aligned}
$$
and, because of $\lim\limits_{L\to\infty}\delta L=C$, representation \eqref{SP12.2} follows.
\end{proof}

\section{A solution to the control problem and its
properties}\label{Solution} In this section we discuss the
solution to the control problem and study its properties. The
functionals $J^{\mathrm{upper}}$ and $J^{\mathrm{lower}}$ are correspondingly given by \eqref{SP11.4}
and \eqref{SP12.1}, and the last terms in these functionals are correspondingly given by
\eqref{SP11.5}
and \eqref{SP12.2}. For our further analysis we need in other representations for these last terms.

Denote
\begin{equation}\label{SCP3}
\psi(C)=\lim_{L\to\infty}\dfrac{\sum\limits_{j=0}^{L-1}c_{L-j}\left(1-\dfrac{2C\mathsf{E}\varsigma}
{(\widetilde{\rho}_{1,2}(\mathsf{E}\varsigma)^{3}+\mathsf{E}\varsigma^{2}-\mathsf{E}\varsigma)L}\right)^j}
{\sum\limits_{j=0}^{L-1}\left(1-\dfrac{2C\mathsf{E}\varsigma}
{(\widetilde{\rho}_{1,2}(\mathsf{E}\varsigma)^{3}+\mathsf{E}\varsigma^{2}-\mathsf{E}\varsigma)L}\right)^j},
\end{equation}
and
\begin{equation}\label{SCP4}
\eta(C)=\lim\limits_{L\to\infty}\dfrac{\sum\limits_{j=0}^{L-1}c_{L-j}\left(1+\dfrac{2C\mathsf{E}\varsigma}
{(\widetilde{\rho}_{1,2}(\mathsf{E}\varsigma)^{3}+\mathsf{E}\varsigma^{2}-\mathsf{E}\varsigma)L}\right)^j}
{\sum\limits_{j=0}^{L-1}\left(1+\dfrac{2C\mathsf{E}\varsigma}
{(\widetilde{\rho}_{1,2}(\mathsf{E}\varsigma)^{3}+\mathsf{E}\varsigma^{2}-\mathsf{E}\varsigma)L}\right)^j}.
\end{equation}
Since $\{c_i\}$ is a bounded sequence, then the limits of
\eqref{SCP3} and \eqref{SCP4} do exist.

The relations between $c^{\mathrm{upper}}$ and $\psi(C)$ and,
respectively, between $c^{\mathrm{lower}}$ and $\eta(C)$ are given
in the lemma below.

\begin{lem}\label{lem5} We have:
\begin{equation}\label{SCP5}
c^{\mathrm{upper}}=\psi(C),
\end{equation}
and
\begin{equation}\label{SCP6}
c^{\mathrm{lower}}=\eta(C).
\end{equation}
\end{lem}

\begin{proof}
From \eqref{SCP3} and \eqref{SCP4} we correspondingly have the
representations:
\begin{equation}\label{SCP7}
\begin{aligned}
&\lim_{L\to\infty}\frac{1}{L}\sum_{j=0}^{L-1}
c_{L-j}\left(1-\frac{2C\mathsf{E}\varsigma}
{(\widetilde{\rho}_{1,2}(\mathsf{E}\varsigma)^{3}+\mathsf{E}\varsigma^{2}-\mathsf{E}\varsigma)L}\right)^j\\&=\psi(C)
\lim_{L\to\infty}\frac{1}{L}\sum_{j=0}^{L-1}
\left(1-\frac{2C\mathsf{E}\varsigma}
{(\widetilde{\rho}_{1,2}(\mathsf{E}\varsigma)^{3}+\mathsf{E}\varsigma^{2}-\mathsf{E}\varsigma)L}\right)^j,
\end{aligned}
\end{equation}
and
\begin{equation}\label{SCP8}
\begin{aligned}
&\lim_{L\to\infty}\frac{1}{L}\sum_{j=0}^{L-1}
c_{L-j}\left(1+\frac{2C\mathsf{E}\varsigma}
{(\widetilde{\rho}_{1,2}(\mathsf{E}\varsigma)^{3}+\mathsf{E}\varsigma^{2}-\mathsf{E}\varsigma)L}\right)^j\\&=\eta(C)
\lim_{L\to\infty}\frac{1}{L}\sum_{j=0}^{L-1}
\left(1+\frac{2C\mathsf{E}\varsigma}
{(\widetilde{\rho}_{1,2}(\mathsf{E}\varsigma)^{3}+\mathsf{E}\varsigma^{2}-\mathsf{E}\varsigma)L}\right)^j.
\end{aligned}
\end{equation}
The desired results follow by direct transformations of the
corresponding right-hand sides of \eqref{SCP7} and \eqref{SCP8}.

Indeed, for the right-hand side of \eqref{SCP7} we
obtain:
\begin{equation}\label{SCP9}
\begin{aligned}
&\psi(C) \lim_{L\to\infty}\frac{1}{L}\sum_{j=0}^{L-1}
\left(1-\frac{2C\mathsf{E}\varsigma}
{(\widetilde{\rho}_{1,2}(\mathsf{E}\varsigma)^{3}+\mathsf{E}\varsigma^{2}-\mathsf{E}\varsigma)L}\right)^j\\&=\psi(C)
\left[1-\exp\left(-\frac{2C\mathsf{E}\varsigma}
{(\widetilde{\rho}_{1,2}(\mathsf{E}\varsigma)^{3}+\mathsf{E}\varsigma^{2}-\mathsf{E}\varsigma)L}\right)\right]
\frac
{(\widetilde{\rho}_{1,2}(\mathsf{E}\varsigma)^{3}+\mathsf{E}\varsigma^{2}-\mathsf{E}\varsigma)L}{2C\mathsf{E}\varsigma}.
\end{aligned}
\end{equation}
On the other hand, from \eqref{SP11.5} we have:
\begin{equation}\label{SCP10}
\begin{aligned}
&c^{\mathrm{upper}}\left[1-\exp\left(-\frac{2C\mathsf{E}\varsigma}
{\widetilde{\rho}_{1,2}(\mathsf{E}\varsigma)^{3}+\mathsf{E}\varsigma^{2}-\mathsf{E}\varsigma}\right)\right]
\frac
{\widetilde{\rho}_{1,2}(\mathsf{E}\varsigma)^{3}+\mathsf{E}\varsigma^{2}-\mathsf{E}\varsigma}{2C\mathsf{E}\varsigma}\\&=
\lim_{L\to\infty}\frac{1}{L}\sum_{j=0}^{L-1}
c_{L-j}\left(1-\frac{2C\mathsf{E}\varsigma}
{(\widetilde{\rho}_{1,2}(\mathsf{E}\varsigma)^{3}+\mathsf{E}\varsigma^{2}-\mathsf{E}\varsigma)L}\right)^j.
\end{aligned}
\end{equation}
Hence, from \eqref{SCP7}, \eqref{SCP9} and \eqref{SCP10} we obtain
\eqref{SCP3}. The proof of \eqref{SCP6} is completely analogous
and uses the representations of \eqref{SP12.2} and \eqref{SCP8}.
\end{proof}

The next lemma establishes the main properties of functions
$\psi(C)$ and $\eta(C)$.

\begin{lem}\label{lem6}
The function $\psi(C)$ is a non-increasing function, and its maximum
is $\psi(0)=c^*$. The function $\eta(C)$ is a non-decreasing
function, and its minimum is $\eta(0)=c^*$. (Recall that
$c^*=\lim\limits_{L\to\infty}\frac{1}{L}\sum\limits_{i=1}^L c_i$ is
defined in Proposition \ref{prop1}.)
\end{lem}

\begin{proof} Let us first prove that $\psi(0)=c^*$ is a maximum
of $\psi(C)$. For this purpose we use the following well-known
inequality (e.g. Hardy, Littlewood and Polya \cite{Hardy
Littlewood Polya 1952} or Marschall and Olkin \cite{Marschall
Olkin 1979}). Let $\{a_n\}$ and $\{b_n\}$ be arbitrary sequences,
one of them is increasing and another decreasing. Then for any
finite sum we have
\begin{equation}\label{SCP11}
\sum_{n=1}^l a_nb_n\leq\frac{1}{l}\sum_{n=1}^l a_n\sum_{n=1}^l
b_n.
\end{equation}

Applying inequality \eqref{SCP11} to the finite sums of the
left-hand side of \eqref{SCP7} and passing to the limit as
$L\to\infty$, we have
\begin{equation}\label{SCP12}
\begin{aligned}
&\lim_{L\to\infty}\frac{1}{L}\sum_{j=0}^{L-1}
c_{L-j}\left(1-\frac{2C\mathsf{E}\varsigma}
{(\widetilde{\rho}_{1,2}(\mathsf{E}\varsigma)^{3}+\mathsf{E}\varsigma^{2}-\mathsf{E}\varsigma)L}\right)^j\\
&\leq \lim_{L\to\infty}\frac{1}{L}\sum_{j=0}^{L-1}
c_{L-j}\lim_{L\to\infty}\frac{1}{L}\sum_{j=0}^{L-1}
\left(1-\frac{2C\mathsf{E}\varsigma}
{(\widetilde{\rho}_{1,2}(\mathsf{E}\varsigma)^{3}+\mathsf{E}\varsigma^{2}-\mathsf{E}\varsigma)L}\right)^j\\
&=\psi(0)\lim_{L\to\infty}\frac{1}{L}\sum_{j=0}^{L-1}
\left(1-\frac{2C\mathsf{E}\varsigma}
{(\widetilde{\rho}_{1,2}(\mathsf{E}\varsigma)^{3}+\mathsf{E}\varsigma^{2}-\mathsf{E}\varsigma)L}\right)^j.
\end{aligned}
\end{equation}
Then, comparing \eqref{SCP7} with \eqref{SCP12} enables us to
conclude,
\begin{equation*}\label{SCP13}
\psi(0)=c^*\geq\psi(C),
\end{equation*}
i.e. $\psi(0)=c^*$ is the maximum value of $\psi(C)$.

Prove now, that $\psi(C)$ is a not increasing function, i.e. for
any nonnegative $C_1\leq C$ we have $\psi(C)\leq\psi(C_1)$.

To prove this note, that for small positive $\delta_1$ and
$\delta_2$ we have (1-$\delta_1$-$\delta_2$) = (1-$\delta_1$)
(1-$\delta_2$) + $O(\delta_1\delta_2)$. Using this idea, one can
prove the monotonicity of $\psi(C)$ by the replacement
\begin{equation}\label{SCP14.0}
\begin{aligned}
&1-\frac{2C\mathsf{E}\varsigma}
{(\widetilde{\rho}_{1,2}(\mathsf{E}\varsigma)^{3}+\mathsf{E}\varsigma^{2}-\mathsf{E}\varsigma)L}\\&=
\left(1-\frac{2C_1\mathsf{E}\varsigma}
{(\widetilde{\rho}_{1,2}(\mathsf{E}\varsigma)^{3}+\mathsf{E}\varsigma^{2}-\mathsf{E}\varsigma)L}\right)
\left(1-\frac{2(C-C_1)\mathsf{E}\varsigma}
{(\widetilde{\rho}_{1,2}(\mathsf{E}\varsigma)^{3}+\mathsf{E}\varsigma^{2}-\mathsf{E}\varsigma)L}\right)\\
&\ \ \ +O\left(\frac{1}{L^2}\right), \ \ C>C_1
\end{aligned}
\end{equation}
in the above asymptotic relations for large $L$. Indeed, notice
that
\begin{equation}\label{SCP14.1}
\begin{aligned}
&\lim_{L\to\infty}\frac{1}{L}\sum_{j=0}^{L-1}
\left(1-\frac{2C_1\mathsf{E}\varsigma}
{(\widetilde{\rho}_{1,2}(\mathsf{E}\varsigma)^{3}+\mathsf{E}\varsigma^{2}-\mathsf{E}\varsigma)L}\right)^j\\
&\ \ \ \times \left(1-\frac{2(C-C_1)\mathsf{E}\varsigma}
{(\widetilde{\rho}_{1,2}(\mathsf{E}\varsigma)^{3}+\mathsf{E}\varsigma^{2}-\mathsf{E}\varsigma)L}\right)^j\\
&= \lim_{L\to\infty}\frac{1}{L}\sum_{j=0}^{L-1}
\left(1-\frac{2C_1\mathsf{E}\varsigma}
{(\widetilde{\rho}_{1,2}(\mathsf{E}\varsigma)^{3}+\mathsf{E}\varsigma^{2}-\mathsf{E}\varsigma)L}\right)^j\\
&\ \ \ \times \lim_{L\to\infty}\frac{1}{L}\sum_{j=0}^{L-1}
\left(1-\frac{2(C-C_1)\mathsf{E}\varsigma}
{(\widetilde{\rho}_{1,2}(\mathsf{E}\varsigma)^{3}+\mathsf{E}\varsigma^{2}-\mathsf{E}\varsigma)L}\right)^j
\end{aligned}
\end{equation}
Therefore, for any non-decreasing sequence $a_j$
\begin{equation}\label{SCP14.2}
\begin{aligned}
&\lim_{L\to\infty}\frac{1}{L}\sum_{j=0}^{L-1}
a_j\left(1-\frac{2C_1\mathsf{E}\varsigma}
{(\widetilde{\rho}_{1,2}(\mathsf{E}\varsigma)^{3}+\mathsf{E}\varsigma^{2}-\mathsf{E}\varsigma)L}\right)^j\\
&\ \ \ \times\left(1-\frac{2(C-C_1)\mathsf{E}\varsigma}
{(\widetilde{\rho}_{1,2}(\mathsf{E}\varsigma)^{3}+\mathsf{E}\varsigma^{2}-\mathsf{E}\varsigma)L}\right)^j\\
&\leq \lim_{L\to\infty}\frac{1}{L}\sum_{j=0}^{L-1}a_j
\left(1-\frac{2C_1\mathsf{E}\varsigma}
{(\widetilde{\rho}_{1,2}(\mathsf{E}\varsigma)^{3}+\mathsf{E}\varsigma^{2}-\mathsf{E}\varsigma)L}\right)^j\\
&\ \ \ \times\lim_{L\to\infty}\frac{1}{L}\sum_{j=0}^{L-1}
\left(1-\frac{2(C-C_1)\mathsf{E}\varsigma}
{(\widetilde{\rho}_{1,2}(\mathsf{E}\varsigma)^{3}+\mathsf{E}\varsigma^{2}-\mathsf{E}\varsigma)L}\right)^j.
\end{aligned}
\end{equation}
Indeed, assume for contrary that
\begin{equation}\label{SCP14.2+}
\begin{aligned}
&\lim_{L\to\infty}\frac{1}{L}\sum_{j=0}^{L-1}
a_j\left(1-\frac{2C_1\mathsf{E}\varsigma}
{(\widetilde{\rho}_{1,2}(\mathsf{E}\varsigma)^{3}+\mathsf{E}\varsigma^{2}-\mathsf{E}\varsigma)L}\right)^j\\
&\ \ \ \times\left(1-\frac{2(C-C_1)\mathsf{E}\varsigma}
{(\widetilde{\rho}_{1,2}(\mathsf{E}\varsigma)^{3}+\mathsf{E}\varsigma^{2}-\mathsf{E}\varsigma)L}\right)^j\\
&> \lim_{L\to\infty}\frac{1}{L}\sum_{j=0}^{L-1}a_j
\left(1-\frac{2C_1\mathsf{E}\varsigma}
{(\widetilde{\rho}_{1,2}(\mathsf{E}\varsigma)^{3}+\mathsf{E}\varsigma^{2}-\mathsf{E}\varsigma)L}\right)^j\\
&\ \ \ \times\lim_{L\to\infty}\frac{1}{L}\sum_{j=0}^{L-1}
\left(1-\frac{2(C-C_1)\mathsf{E}\varsigma}
{(\widetilde{\rho}_{1,2}(\mathsf{E}\varsigma)^{3}+\mathsf{E}\varsigma^{2}-\mathsf{E}\varsigma)L}\right)^j.
\end{aligned}
\end{equation}
Then, applying inequality \eqref{SCP11} to the right-hand side of
\eqref{SCP14.2+}, we obtain:
\begin{equation}\label{SCP14.2++}
\begin{aligned}
&\lim_{L\to\infty}\frac{1}{L}\sum_{j=0}^{L-1}a_j
\left(1-\frac{2C_1\mathsf{E}\varsigma}
{(\widetilde{\rho}_{1,2}(\mathsf{E}\varsigma)^{3}+\mathsf{E}\varsigma^{2}-\mathsf{E}\varsigma)L}\right)^j\\
&\ \ \ \times\lim_{L\to\infty}\frac{1}{L}\sum_{j=0}^{L-1}
\left(1-\frac{2(C-C_1)\mathsf{E}\varsigma}
{(\widetilde{\rho}_{1,2}(\mathsf{E}\varsigma)^{3}+\mathsf{E}\varsigma^{2}-\mathsf{E}\varsigma)L}\right)^j\\
&\leq\lim_{L\to\infty}\frac{1}{L}\sum_{j=0}^{L-1}a_j\lim_{L\to\infty}\frac{1}{L}\sum_{j=0}^{L-1}
\left(1-\frac{2C_1\mathsf{E}\varsigma}
{(\widetilde{\rho}_{1,2}(\mathsf{E}\varsigma)^{3}+\mathsf{E}\varsigma^{2}-\mathsf{E}\varsigma)L}\right)^j\\
&\ \ \ \times\lim_{L\to\infty}\frac{1}{L}\sum_{j=0}^{L-1}
\left(1-\frac{2(C-C_1)\mathsf{E}\varsigma}
{(\widetilde{\rho}_{1,2}(\mathsf{E}\varsigma)^{3}+\mathsf{E}\varsigma^{2}-\mathsf{E}\varsigma)L}\right)^j\\
&=\lim_{L\to\infty}\frac{1}{L}\sum_{j=0}^{L-1}a_j\lim_{L\to\infty}\frac{1}{L}\sum_{j=0}^{L-1}
\left(1-\frac{2C\mathsf{E}\varsigma}
{(\widetilde{\rho}_{1,2}(\mathsf{E}\varsigma)^{3}+\mathsf{E}\varsigma^{2}-\mathsf{E}\varsigma)L}\right)^j.
\end{aligned}
\end{equation}
Since the left-hand side of \eqref{SCP14.2+} is
$$
\lim_{L\to\infty}\frac{1}{L}\sum_{j=0}^{L-1}
a_j\left(1-\frac{2C\mathsf{E}\varsigma}
{(\widetilde{\rho}_{1,2}(\mathsf{E}\varsigma)^{3}+\mathsf{E}\varsigma^{2}-\mathsf{E}\varsigma)L}\right)^j
$$
(see relation \eqref{SCP14.0}), then comparison of the last
obtained term in \eqref{SCP14.2++} with the left-hand side of
\eqref{SCP14.2+} enables us to write:
\begin{equation*}\label{SCP14.2+++}
\begin{aligned}
&\lim_{L\to\infty}\frac{1}{L}\sum_{j=0}^{L-1}
a_j\left(1-\frac{2C\mathsf{E}\varsigma}
{(\widetilde{\rho}_{1,2}(\mathsf{E}\varsigma)^{3}+\mathsf{E}\varsigma^{2}-\mathsf{E}\varsigma)L}\right)^j\\
&>\lim_{L\to\infty}\frac{1}{L}\sum_{j=0}^{L-1}a_j\lim_{L\to\infty}\frac{1}{L}\sum_{j=0}^{L-1}
\left(1-\frac{2C\mathsf{E}\varsigma}
{(\widetilde{\rho}_{1,2}(\mathsf{E}\varsigma)^{3}+\mathsf{E}\varsigma^{2}-\mathsf{E}\varsigma)L}\right)^j.
\end{aligned}
\end{equation*}
The contradiction with basic inequality \eqref{SCP11} proves
\eqref{SCP14.2}.

Taking into account \eqref{SCP14.1} and \eqref{SCP14.2}, the
extended version of \eqref{SCP12} after an application of
\eqref{SCP11} now looks
\begin{equation}\label{SCP14}
\begin{aligned}
&\lim_{L\to\infty}\frac{1}{L}\sum_{j=0}^{L-1}
c_{L-j}\left(1-\frac{2C\mathsf{E}\varsigma}
{(\widetilde{\rho}_{1,2}(\mathsf{E}\varsigma)^{3}+\mathsf{E}\varsigma^{2}-\mathsf{E}\varsigma)L}\right)^j\\
&=\lim_{L\to\infty}\frac{1}{L}\sum_{j=0}^{L-1}
c_{L-j}\left(1-\frac{2C_1\mathsf{E}\varsigma}
{(\widetilde{\rho}_{1,2}(\mathsf{E}\varsigma)^{3}+\mathsf{E}\varsigma^{2}-\mathsf{E}\varsigma)L}\right)^j\\
&\ \ \ \times\left(1-\frac{2(C-C_1)\mathsf{E}\varsigma}
{(\widetilde{\rho}_{1,2}(\mathsf{E}\varsigma)^{3}+\mathsf{E}\varsigma^{2}-\mathsf{E}\varsigma)L}\right)^j\\
&\leq \lim_{L\to\infty}\frac{1}{L}\sum_{j=0}^{L-1}
c_{L-j}\left(1-\frac{2C_1\mathsf{E}\varsigma}
{(\widetilde{\rho}_{1,2}(\mathsf{E}\varsigma)^{3}+\mathsf{E}\varsigma^{2}-\mathsf{E}\varsigma)L}\right)^j\\
&\ \ \ \times\lim_{L\to\infty}\frac{1}{L}\sum_{j=0}^{L-1}
\left(1-\frac{2(C-C_1)\mathsf{E}\varsigma}
{(\widetilde{\rho}_{1,2}(\mathsf{E}\varsigma)^{3}+\mathsf{E}\varsigma^{2}-\mathsf{E}\varsigma)L}\right)^j\\
&=\psi(C_1)\lim_{L\to\infty}\frac{1}{L}\sum_{j=0}^{L-1}
\left(1-\frac{2C_1\mathsf{E}\varsigma}
{(\widetilde{\rho}_{1,2}(\mathsf{E}\varsigma)^{3}+\mathsf{E}\varsigma^{2}-\mathsf{E}\varsigma)L}\right)^j\\
&\ \ \ \times\lim_{L\to\infty}\frac{1}{L}\sum_{j=0}^{L-1}
\left(1-\frac{2(C-C_1)\mathsf{E}\varsigma}
{(\widetilde{\rho}_{1,2}(\mathsf{E}\varsigma)^{3}+\mathsf{E}\varsigma^{2}-\mathsf{E}\varsigma)L}\right)^j.
\end{aligned}
\end{equation}
On the other hand, the right-hand side of \eqref{SCP7} can be
rewritten
\begin{equation}\label{SCP15}
\begin{aligned}
&\psi(C) \lim_{L\to\infty}\frac{1}{L}\sum_{j=0}^{L-1}
\left(1-\frac{2C\mathsf{E}\varsigma}
{(\widetilde{\rho}_{1,2}(\mathsf{E}\varsigma)^{3}+\mathsf{E}\varsigma^{2}-\mathsf{E}\varsigma)L}\right)^j\\
&=\psi(C) \lim_{L\to\infty}\frac{1}{L}\sum_{j=0}^{L-1}
\left(1-\frac{2C_1\mathsf{E}\varsigma}
{(\widetilde{\rho}_{1,2}(\mathsf{E}\varsigma)^{3}+\mathsf{E}\varsigma^{2}-\mathsf{E}\varsigma)L}\right)^j\\
& \ \ \ \times\left(1-\frac{2(C-C_1)\mathsf{E}\varsigma}
{(\widetilde{\rho}_{1,2}(\mathsf{E}\varsigma)^{3}+\mathsf{E}\varsigma^{2}-\mathsf{E}\varsigma)L}\right)^j\\
&=\psi(C)\lim_{L\to\infty}\frac{1}{L}\sum_{j=0}^{L-1}
\left(1-\frac{2C_1\mathsf{E}\varsigma}
{(\widetilde{\rho}_{1,2}(\mathsf{E}\varsigma)^{3}+\mathsf{E}\varsigma^{2}-\mathsf{E}\varsigma)L}\right)^j\\
&\ \ \ \times\lim_{L\to\infty}\frac{1}{L}\sum_{j=0}^{L-1}
\left(1-\frac{2(C-C_1)\mathsf{E}\varsigma}
{(\widetilde{\rho}_{1,2}(\mathsf{E}\varsigma)^{3}+\mathsf{E}\varsigma^{2}-\mathsf{E}\varsigma)L}\right)^j.
\end{aligned}
\end{equation}
The last equality in \eqref{SCP15} is the application of
\eqref{SCP14.1}. From \eqref{SCP14} and \eqref{SCP15} we finally
obtain $\psi(C_1)\leq\psi(C)$ for any positive $C_1\geq C$.

The
first statement of Lemma \ref{lem6} is proved. The proof of the
second statement of this lemma is similar.
\end{proof}

In the following we need in a stronger result that is given by Lemma
\ref{lem6}. Namely, we will prove the following lemma.

\begin{lem}\label{lem7}
If the sequence $\{c_i\}$ contains at least two distinct values,
then the function $\psi(C)$ is a strictly decreasing function, and
the function $\eta(C)$ is a strictly increasing function.
\end{lem}

\begin{proof} In order to prove this lemma it is sufficient to prove
that if the sequence $\{c_i\}$ is nontrivial, that is there are at
least two distinct values of this sequence, then for any distinct
real numbers $C_1\neq C_2$ the values of functions are also
distinct, that is, $\psi(C_1)\neq \psi(C_2)$ and $\eta(C_1)\neq
\eta(C_2)$. Let us prove the first inequality: $\psi(C_1)\neq
\psi(C_2)$. Rewrite \eqref{SCP3} as
\begin{equation}\label{SCP3'}
\psi(C)=\lim_{L\to\infty}\dfrac{\dfrac{1}{L}\sum\limits_{j=0}^{L-1}c_{L-j}\left(1-\dfrac{2C\mathsf{E}\varsigma}
{(\widetilde{\rho}_{1,2}(\mathsf{E}\varsigma)^{3}+\mathsf{E}\varsigma^{2}-\mathsf{E}\varsigma)L}\right)^j}
{\dfrac{1}{L}\sum\limits_{j=0}^{L-1}\left(1-\dfrac{2C\mathsf{E}\varsigma}
{(\widetilde{\rho}_{1,2}(\mathsf{E}\varsigma)^{3}+\mathsf{E}\varsigma^{2}-\mathsf{E}\varsigma)L}\right)^j}.
\end{equation}
The limit of the denominator is equal to
$\exp\left(-\frac{2C\mathsf{E}\varsigma}
{\widetilde{\rho}_{1,2}(\mathsf{E}\varsigma)^{3}+\mathsf{E}\varsigma^{2}-\mathsf{E}\varsigma}\right)$.
The limit of the numerator does exist and bounded, since the
sequence $\{c_i\}$ is assumed to be bounded. As well, according to
the other representation following from Lemma \ref{lem5} and
relation \eqref{SP11.5}, it is an analytic function in $C$ taking a
nontrivial set of values.

The  analytic function $\psi(C)$ is defined for all real $C\geq0$
and it can be extended analytically for the whole complex plane.
Namely, for real negative values $C$ we arrive at the function
$\eta(C)=\psi(-C)$. According to the maximum principle for the
module of an analytic function, if an analytic function takes the
same values in two distinct points inside a domain, then the
function must be the constant. If $c_i=c^*$ for all $i=1,2,\ldots$,
then this is just the case where $\psi(C)=c^*$ for all $C$. If there
exist $i_0$ and $i_1$ such that $c_{i_0}\neq c_{i_1}$, then the
function $\psi(C)$ cannot be a constant, because the analytic
function is uniquely defined by the coefficients in the series
expansion.
 So, the inequality $\psi(C_1)\neq
\psi(C_2)$ for distinct values $C_1$ and $C_2$ follows. The proof of
the second inequality $\eta(C_1)\neq \eta(C_2)$ is similar.
\end{proof}

We are ready now to formulate and prove a main theorem on optimal
control of the dam model considered in the present paper.

\begin{thm}\label{thm3}
Under the assumption that the costs $c_i$ are not increasing, and under additional mild conditions of Theorems \ref{thm1} and \ref{thm4}, there is a unique solution to the control problem. The solution to the control problem is defined by choice of the parameter $\rho_1$ as follows.

Let $\overline{C}$ be the minimum value of the functional $J^{\mathrm{upper}}$ defined in \eqref{SP11.4} and \eqref{SP11.5} and, respectively, let $\underline{C}$ be the minimum value of the functional $J^{\mathrm{lower}}$ defined in \eqref{SP12.1} and \eqref{SP12.2}. Then at least one of two parameters $\overline{C}$ or $\underline{C}$ must be equal to zero.

(1) In the case $\overline{C}=0$ and $\underline{C}>0$, the solution to the control problem is achieved for $\rho_1=1-\delta$, where positive $\delta$ vanishes such that $\delta L\to \underline{C}$ as $L\to\infty$.

(2) In the case $\underline{C}=0$ and $\overline{C}>0$, the solution to the control problem is achieved for
$\rho_1=1+\delta$, where positive $\delta$ vanishes such that $\delta L\to {\overline{C}}$ as $L\to\infty$.

(3) In the case where both $\overline{C}=0$ and $\underline{C}=0$, the solution to the control problem is $\rho_1=1$.
\end{thm}

\begin{proof}
Note first, that under the assumptions made there is a unique
solution to the control problem considered in this paper. Indeed, a
solution contains two terms one of them corresponds to the
expression for $p_1J_1+p_2J_2$ in \eqref{I1} and another one
corresponds to the term
$\sum\limits_{i=L^{\mathrm{lower}}+1}^{L^{\mathrm{upper}}}c_iq_i$ in
\eqref{I1}. The first term of a solution is related to the models
where the water costs are not taken into account, while the
additional second term is related to the extended problem, where the
water costs are taken into account.

In the case where the water costs are not taken into account, the
existence of a unique solution to the control problem for the
particular system in \cite{Abramov 2007} follows from the main
result of that paper. The same result holds true for a more general
model with compound Poisson input flow but without water costs
included. The last is supported by Theorems \ref{i1} - \ref{i4},
which are similar to those Theorems 4.1 - 4.4 of \cite{Abramov
2007}.

In the case of the dam model, where the water costs are taken into
account, the second term in the solution is either
$c^{\mathrm{upper}}$ or $c^{\mathrm{lower}}$. According to Lemma
\ref{lem5} $c^{\mathrm{upper}}=\psi(C)$ and
$c^{\mathrm{lower}}=\eta(C)$, and according to Lemmas \ref{lem6} and
\ref{lem7} the function $\psi(C)$ is strictly decreasing in $C$,
while the function $\eta(C)$ is strictly increasing in $C$, and
$\psi(0)=\eta(0)=c^*$. According to these properties, there is a
unique solution to the control problem considered in the present
paper as well, and it satisfies the following properties.

In the case where the both minima of $J^{\mathrm{upper}}$ and $J^{\mathrm{lower}}$
are achieved in $C=0$, that is both $\overline{C}=0$ and $\underline{C}=0$, then  $c^{\mathrm{upper}}=c^{\mathrm{lower}}=c^*$ and the term $p_1J_1+p_2J_2$ of the objective function in \eqref{I1} coincides with the term
\begin{equation*}
j_1\frac
{\widetilde{\rho}_{1,2}(\mathsf{E}\varsigma)^{3}+\mathsf{E}\varsigma^{2}-\mathsf{E}\varsigma}
{2\mathsf{E}\varsigma} +j_2\frac{\rho_2}{1-\rho_2}\cdot \frac
{\widetilde{\rho}_{1,2}(\mathsf{E}\varsigma)^{3}+\mathsf{E}\varsigma^{2}-\mathsf{E}\varsigma}
{2\mathsf{E}\varsigma}
\end{equation*}
in \eqref{SP11.1}. That is both the minimum of $J^{\mathrm{upper}}$ and that of $J^{\mathrm{lower}}$ are the same, and they are equal to the
right-hand side of \eqref{SP11.1}. In this case the minimum of the objective
function in \eqref{I1} is achieved for $\rho_1=1$.

If the minimum of $J^{\mathrm{lower}}$ is achieved for
$C=\underline{C}>0$, then, since $\eta(C)$ is strictly increasing,
we have $c^{\mathrm{lower}}> c^*$,  and hence the term
$p_1J_1+p_2J_2$ of the objective function in \eqref{I1} satisfies
the inequality:
\begin{equation*}
p_1J_1+p_2J_2< j_1\frac
{\widetilde{\rho}_{1,2}(\mathsf{E}\varsigma)^{3}+\mathsf{E}\varsigma^{2}-\mathsf{E}\varsigma}
{2\mathsf{E}\varsigma}+j_2\frac{\rho_2}{1-\rho_2}\cdot \frac
{\widetilde{\rho}_{1,2}(\mathsf{E}\varsigma)^{3}+\mathsf{E}\varsigma^{2}-\mathsf{E}\varsigma}
{2\mathsf{E}\varsigma}.
\end{equation*}
This implies that $J^{\mathrm{lower}}$ is less than the right-hand
side of \eqref{SP11.1}. As well, in this case $c^{\mathrm{upper}}<
c^*$, and the minimum of $J^{\mathrm{upper}}$ must be achieved for
$C=\overline{C}=0$.
 In this case the minimum
of the objective function in \eqref{I1} is achieved for $\rho_1=1-\delta$, where positive
$\delta$ vanishes as $L\to\infty$, and
$L\delta\to\underline{C}$.

In the opposite case, if the minimum of $J^{\mathrm{upper}}$ is
achieved for $C=\overline{C}>0$, then the arguments are similar to
those above, and $J^{\mathrm{upper}}$ is not greater than the
right-hand side of \eqref{SP11.1}. The minimum of
$J^{\mathrm{lower}}$ must be achieved for $C=\overline{C}=0$. In
this case the minimum of the objective function in \eqref{I1} is
achieved for $\rho_1=1+\delta$, where positive $\delta$ vanishes as
$L\to\infty$, and $L\delta\to\overline{C}$.
\end{proof}

From Theorem \ref{thm3} we have the following evident property of
the optimal control.

\begin{cor}\label{cor2}
The solution to the control problem can be $\rho_1=1$ only in the
case $j_1\leq j_2\frac{\rho_2}{1-\rho_2}$. Specifically, the
equality is achieved only for $c_i\equiv c$, $i=1,2,\ldots,L$,
where $c$ is an arbitrary positive constant.
\end{cor}

Although Corollary \ref{cor2} provides a result in the form of
simple inequality, this result is not really useful, since it is an
evident extension of the result of \cite{Abramov 2007}. A more
constructive result is obtained for the special case considered in
the next section.

\section{Example of linear costs}\label{Examples}

In this section we study an example related to the case of linear
costs.

Assume that $c_1$ and $c_L<c_1$ are given.  Then the assumption
that the costs are linear means, that
\begin{equation}\label{E1}
c_i = c_1 - \frac{i-1}{L-1}(c_1-c_L), \ \ i=1,2,\ldots, L.
\end{equation}
It is assumed that as $L$ is changed, the costs are recalculated
as follows. The first and last values of the cost $c_1$ and $c_L$
remains the same. Other costs in the intermediate points are
recalculated according to \eqref{E1}.

Therefore, to avoid confusing with the appearance of the index $L$
for the fixed (unchangeable) values of cost  $c_1$ and $c_L$, we
use the other notation: $c_1=\overline{c}$ and
$c_L=\underline{c}$. Then \eqref{E1} has the form
\begin{equation}\label{E2}
c_i = \overline{c} - \frac{i-1}{L-1}(\overline{c}-\underline{c}),
\ \ i=1,2,\ldots, L.
\end{equation}

In the following we shall also use the inverse form of \eqref{E2}.
Namely,
\begin{equation}\label{E3}
c_{L-i}=\underline{c}+\frac{i}{L-1}(\overline{c}-\underline{c}), \
\ i=0,1,\ldots, L-1.
\end{equation}

Apparently,
\begin{equation}\label{E4}
c^*=\frac{\overline{c}+\underline{c}}{2}.
\end{equation}
For $c^{\mathrm{upper}}$ we have
\begin{equation}\label{E5}
\begin{aligned}
c^{\mathrm{upper}}&=\psi(C)\\
&=\lim\limits_{L\to\infty}\frac{\sum\limits_{j=0}^{L-1}
\left(\underline{c}+\dfrac{j}{L-1}(\overline{c}-\underline{c})\right)
\left(1-\dfrac{2C\mathsf{E}\varsigma}
{(\widetilde{\rho}_{1,2}(\mathsf{E}\varsigma)^{3}+\mathsf{E}\varsigma^{2}-\mathsf{E}\varsigma)L}\right)^j}
{\sum\limits_{j=0}^{L-1} \left(1-\dfrac{2C\mathsf{E}\varsigma}
{(\widetilde{\rho}_{1,2}(\mathsf{E}\varsigma)^{3}+\mathsf{E}\varsigma^{2}-\mathsf{E}\varsigma)L}\right)^j}\\
&=\underline{c}+(\overline{c}-\underline{c})\lim_{L\to\infty}\frac{1}{L-1}\cdot
\frac{\sum\limits_{j=0}^{L-1} j
\left(1-\dfrac{2C\mathsf{E}\varsigma}
{(\widetilde{\rho}_{1,2}(\mathsf{E}\varsigma)^{3}+\mathsf{E}\varsigma^{2}-\mathsf{E}\varsigma)L}\right)^j}
{\sum\limits_{j=0}^{L-1} \left(1-\dfrac{2C\mathsf{E}\varsigma}
{(\widetilde{\rho}_{1,2}(\mathsf{E}\varsigma)^{3}+\mathsf{E}\varsigma^{2}-\mathsf{E}\varsigma)L}\right)^j}\\
&=\underline{c}+(\overline{c}-\underline{c}) \dfrac
{\widetilde{\rho}_{1,2}(\mathsf{E}\varsigma)^{3}+\mathsf{E}\varsigma^{2}-\mathsf{E}\varsigma}{2C\mathsf{E}\varsigma}\\
&\ \ \ \times \frac{-\dfrac{2C\mathsf{E}\varsigma}
{\widetilde{\rho}_{1,2}(\mathsf{E}\varsigma)^{3}+\mathsf{E}\varsigma^{2}-\mathsf{E}\varsigma}+
\exp\left(\dfrac{2C\mathsf{E}\varsigma}
{\widetilde{\rho}_{1,2}(\mathsf{E}\varsigma)^{3}+\mathsf{E}\varsigma^{2}-\mathsf{E}\varsigma}\right)-1}
{\exp\left(\dfrac{2C\mathsf{E}\varsigma}
{\widetilde{\rho}_{1,2}(\mathsf{E}\varsigma)^{3}+\mathsf{E}\varsigma^{2}-\mathsf{E}\varsigma}\right)-1}.
\end{aligned}
\end{equation}
For example, as $C$ converges to zero in \eqref{E5}, then
$c^{\mathrm{upper}}$ converges to
$\underline{c}+\frac{1}{2}(\overline{c}-\underline{c})=c^*$. This
is in agreement with the statement of Proposition \ref{prop1}.

In turn, for $c^{\mathrm{lower}}$ we have
\begin{equation}\label{E6}
\begin{aligned}
c^{\mathrm{lower}}&=\eta(C)\\
&=\lim_{L\to\infty}\frac{\sum\limits_{j=0}^{L-1}
\left(\underline{c}+\dfrac{j}{L-1}(\overline{c}-\underline{c})\right)
\left(1+\dfrac{2C\mathsf{E}\varsigma}
{(\widetilde{\rho}_{1,2}(\mathsf{E}\varsigma)^{3}+\mathsf{E}\varsigma^{2}-\mathsf{E}\varsigma)L}\right)^j}
{\sum\limits_{j=0}^{L-1}
\left(1+\dfrac{2C\mathsf{E}\varsigma}
{(\widetilde{\rho}_{1,2}(\mathsf{E}\varsigma)^{3}+\mathsf{E}\varsigma^{2}-\mathsf{E}\varsigma)L}\right)^j}\\
&=\underline{c}+(\overline{c}-\underline{c})\lim_{L\to\infty}\frac{1}{L-1}\cdot
\frac{\sum\limits_{j=0}^{L-1} j
\left(1+\dfrac{2C\mathsf{E}\varsigma}
{(\widetilde{\rho}_{1,2}(\mathsf{E}\varsigma)^{3}+\mathsf{E}\varsigma^{2}-\mathsf{E}\varsigma)L}\right)^j}
{\sum\limits_{j=0}^{L-1}
\left(1+\dfrac{2C\mathsf{E}\varsigma}
{(\widetilde{\rho}_{1,2}(\mathsf{E}\varsigma)^{3}+\mathsf{E}\varsigma^{2}-\mathsf{E}\varsigma)L}\right)^j}\\
&=\underline{c}+(\overline{c}-\underline{c}) \frac
{\widetilde{\rho}_{1,2}(\mathsf{E}\varsigma)^{3}+\mathsf{E}\varsigma^{2}-\mathsf{E}\varsigma}
{2C\mathsf{E}\varsigma}\\
&\ \ \ \times \frac{\dfrac{2C\mathsf{E}\varsigma}
{\widetilde{\rho}_{1,2}(\mathsf{E}\varsigma)^{3}+\mathsf{E}\varsigma^{2}-\mathsf{E}\varsigma}-1+
\exp\left(-\dfrac{2C\mathsf{E}\varsigma}
{\widetilde{\rho}_{1,2}(\mathsf{E}\varsigma)^{3}+\mathsf{E}\varsigma^{2}-\mathsf{E}\varsigma}\right)}
{1-\exp\left(-\dfrac{2C\mathsf{E}\varsigma}
{\widetilde{\rho}_{1,2}(\mathsf{E}\varsigma)^{3}+\mathsf{E}\varsigma^{2}-\mathsf{E}\varsigma}\right)}.
\end{aligned}
\end{equation}
Again, as $C$ converges to zero in \eqref{E6}, then $c^{\mathrm{lower}}$
converges to
$\underline{c}+\frac{1}{2}(\overline{c}-\underline{c})=c^*$.
So, we arrive at the agreement with the statement of Proposition
\ref{prop1}.

\smallskip
We cannot give the explicit solution because the calculations are
very routine and cumbersome. However, we provide a numerical result.
For simplicity, the input flow in the numerical example is assumed
to be ordinary Poisson, that is we set $\mathrm{E}\varsigma=1$ and
$\mathrm{E}\varsigma^{2}=1$ in our calculations.

Following Corollary \ref{cor2}, take first
$j_1=j_2\frac{\rho_2}{1-\rho_2}$. Clearly, that for these relation
between parameters $j_1$ and $j_2$ the minimum of $J^{\mathrm{lower}}$ must
be achieved for $C=0$, while the minimum of $J^{\mathrm{upper}}$ must be
achieved for a positive $C$. Now, keeping $j_1$ fixed assume that
$j_2$ increases. Then, the problem is to find the value for
parameter $j_2$ such that the value $C$ corresponding to the
minimization problem of $J^{\mathrm{upper}}$ reaches the point 0.

In our example we take $j_1=1$, $\rho_2=\frac{1}{2}$,
$\underline{c}=1$, $\overline{c}=2$, $\widetilde{\rho}_{1,2}=1$.
In the table below we outline some values $j_2$ and the
corresponding value $C$ for optimal solution of functional
$J^{\mathrm{upper}}$. It is seen from the table that the optimal value is
achieved in the case $j_2\approx1.34$. Therefore, in the present
example $j_1=1$ and $j_2\approx1.34$ lead to the optimal solution
$\rho_1=1$.

\begin{table}
    \begin{center}
        \begin{tabular}{c|c}\hline
         Parameter & Argument of optimal value\\
        $j_2$ & $C$\\
        \hline
         1.06 & 0.200\\
         1.08 & 0.182\\
         1.10 & 0.165\\
         1.12 & 0.149\\
         1.14 & 0.134\\
         1.16 & 0.120\\
        1.18 & 0.104\\
        1.20 & 0.090\\
        1.25 & 0.055\\
        1.30 & 0.022\\
        1.33 & 0.010\\
        1.34 & 0\\
        \hline
        \end{tabular}

        \medskip
        \caption{The values of parameter $j_2$ and corresponding
        arguments of optimal value $C$}
    \end{center}
\end{table}

\section*{Acknowledgements}
The author thanks Prof. Phil Howlett (University of South
Australia), whose questions in a local seminar in the University of
South Australia initiated the solution of this circle of problems
including the earlier paper of the author \cite{Abramov 2007}.


\end{document}